\documentclass[english,11pt]{article}
\usepackage[T1]{fontenc}
\usepackage[latin9]{inputenc}
\usepackage{refstyle}
\usepackage{bm}
\usepackage{amstext}
\usepackage{amssymb}
\usepackage{graphicx}
\usepackage{hyperref}
\hypersetup{
    colorlinks=true,
    linkcolor=blue,
    citecolor=blue,      
    urlcolor=blue,
}
\usepackage{float}
\usepackage{mathrsfs}
\usepackage{amsmath}
\usepackage{amsthm}
\usepackage{amssymb}
\usepackage[numbers]{natbib}
\usepackage{xargs}[2008/03/08]
\usepackage{footmisc}

\usepackage{booktabs}
\usepackage{array}
\usepackage{multirow}
\newcolumntype{C}[1]{>{\centering\arraybackslash$}m{#1}<{$}}

\newcommand{\expm}[2]{#1\mathrm{e}\raisebox{0.3\height}{\scalebox{0.7}{$-$}}#2}
\newcommand{\bmexpm}[2]{\pmb{#1\mathrm{e}\raisebox{0.3\height}{\scalebox{0.7}{$-$}}#2}}

\makeatletter
\AtBeginDocument{\providecommand\propref[1]{\ref{prop:#1}}}
\AtBeginDocument{\providecommand\defref[1]{\ref{def:#1}}}
\AtBeginDocument{\providecommand\asmref[1]{\ref{asm:#1}}}
\AtBeginDocument{\providecommand\lemref[1]{\ref{lem:#1}}}
\AtBeginDocument{\providecommand\appref[1]{\ref{app:#1}}}
\AtBeginDocument{\providecommand\corref[1]{\ref{cor:#1}}}
\AtBeginDocument{\providecommand\thmref[1]{\ref{thm:#1}}}
\AtBeginDocument{\providecommand\claimref[1]{\ref{claim:#1}}}
\RS@ifundefined{subsecref}
  {\newref{subsec}{name = \RSsectxt}}
  {}
\RS@ifundefined{thmref}
  {\def\RSthmtxt{theorem~}\newref{thm}{name = \RSthmtxt}}
  {}
\RS@ifundefined{lemref}
  {\def\RSlemtxt{lemma~}\newref{lem}{name = \RSlemtxt}}
  {}

\theoremstyle{plain}
\newtheorem{theorem}{\protect\theoremname}[section]
\theoremstyle{definition}
\newtheorem{definition}[theorem]{\protect\definitionname}
\theoremstyle{plain}
\newtheorem{proposition}[theorem]{\protect\propositionname}
\theoremstyle{plain}
\newtheorem{lemma}[theorem]{\protect\lemmaname}
\theoremstyle{plain}
\newtheorem{corollary}[theorem]{\protect\corollaryname}
\theoremstyle{remark}
\newtheorem{claim}[theorem]{\protect\claimname}

\usepackage[margin=1in]{geometry}
\newtheorem{assume}{Assumption}
\newref{thm}{name=Theorem~,Name=Theorem~}
\newref{prop}{name=Proposition~,Name=Proposition~}
\newref{def}{name=Definition~,Name=Definition~}
\newref{lem}{name=Lemma~,Name=Lemma~}
\newref{cor}{name=Corollary~,Name=Corollary~}
\newref{asm}{name=Assumption~,Name=Assumption~}
\newref{app}{name=Appendix~,Name=Appendix~}
\newref{claim}{name=Claim~,Name=Claim~}

\makeatother
\usepackage{babel}
\providecommand{\claimname}{Claim}
\providecommand{\corollaryname}{Corollary}
\providecommand{\definitionname}{Definition}
\providecommand{\lemmaname}{Lemma}
\providecommand{\propositionname}{Proposition}
\providecommand{\theoremname}{Theorem}

\begin{document}
\title{Well-conditioned Primal-Dual Interior-point Method for Accurate Low-rank Semidefinite
Programming\thanks{Financial support for this work was provided by NSF CAREER Award ECCS-2047462 and ONR Award N00014-24-1-2671.}}
\date{}
\author{Hong-Ming Chiu\\
University of Illinois Urbana-Champaign\\
\texttt{hmchiu2@illinois.edu}\\
\and 
Richard Y. Zhang\\
University of Illinois Urbana-Champaign\\
\texttt{ryz@illinois.edu}}

\maketitle
\global\long\def\sym{\mathrm{sym}}%
\global\long\def\R{\mathbb{R}}%
\global\long\def\S{\mathcal{S}}%
\global\long\def\rank{\mathrm{rank}\,}%
\global\long\def\A{\mathbf{A}}%
\global\long\def\AA{\mathcal{A}}%
\global\long\def\DD{\mathcal{D}}%
\global\long\def\K{\mathcal{K}}%
\global\long\def\bK{\mathbf{K}}%
\global\long\def\N{\mathcal{N}}%
\global\long\def\Feas{\mathcal{F}}%
\global\long\def\E{\mathbf{E}}%
\global\long\def\H{\mathbf{H}}%
\global\long\def\G{\mathbf{G}}%
\global\long\def\D{\mathbf{D}}%
\global\long\def\C{\mathbf{C}}%
\global\long\def\B{\mathbf{B}}%
\global\long\def\P{\mathbf{P}}%
\global\long\def\Q{\mathbf{Q}}%
\global\long\def\M{\mathbf{M}}%
\global\long\def\U{\mathbf{U}}%
\global\long\def\Sig{\mathbf{\Sigma}}%
\global\long\def\UU{\mathcal{U}}%
\global\long\def\J{\mathbf{J}}%
\global\long\def\macheps{\varepsilon_\mathrm{mach}}%
\global\long\def\vector{\operatorname{vec}}%
\global\long\def\svec{\operatorname{vec}_{\S}}%
\global\long\def\skron{\otimes_{\S}}%
\global\long\def\orth{\operatorname{orth}}%
\global\long\def\cond{\operatorname{cond}}%
\global\long\def\diag{\operatorname{diag}}%
\global\long\def\tr{\operatorname{tr}}%
\global\long\def\forall{\text{for all }}%
\global\long\def\colspan{\operatorname{colsp}}%
\global\long\def\eqdef{\overset{\mathrm{def}}{=}}%
\global\long\def\inner#1#2{\left\langle #1,#2\right\rangle }%
\global\long\def\PCG{\operatorname{PCG}}%
\global\long\def\MINRES{\operatorname{MINRES}}%

\begin{abstract}
We describe how the low-rank structure in an SDP can be exploited to reduce the per-iteration cost of a convex primal-dual interior-point method down to $O(n^{3})$ time and $O(n^{2})$ memory, even at very high accuracies. A traditional difficulty is the dense Newton subproblem at each iteration, which becomes progressively ill-conditioned as progress is made towards the solution. Preconditioners have been proposed to improve conditioning, but these can be expensive to set up, and fundamentally become ineffective at high accuracies, as the preconditioner itself becomes increasingly ill-conditioned. Instead, we present a \emph{well-conditioned reformulation} of the Newton subproblem that is cheap to set up, and whose condition number is guaranteed to remain bounded over all iterations of the interior-point method. In theory, applying an inner iterative method to the reformulation reduces the per-iteration cost of the outer interior-point method to $O(n^{3})$ time and $O(n^{2})$ memory. We also present a \emph{well-conditioned preconditioner} that theoretically increases the outer per-iteration cost to $O(n^{3}r^{3})$ time and $O(n^{2}r^{2})$ memory, where $r$ is an upper-bound on the solution rank, but in practice greatly improves the convergence of the inner iterations. 
\end{abstract}

\section{Introduction}
Given problem data $\AA:\S^{n}\to\R^{m}$ and $b\in\R^{m}$ and $C\in\S^{n}$,
we seek to solve the standard-form semidefinite program (SDP) \begin{subequations}\label{eq:sdp-pd}
\begin{align}
X^{\star} & =\arg\min_{X\in\S^{n}}\{\inner CX:\AA(X)=b,\quad X\succeq0\},\label{eq:sdp}\\
y^{\star},S^{\star} & =\arg\max_{S\in\S^{n},y\in\R^{m}}\{\inner by:\AA^{T}(y)+S=C,\quad S\succeq0\},\label{eq:sdd}
\end{align}
assuming that the primal problem (\ref{eq:sdp}) admits a \emph{unique
low-rank solution}
\begin{equation}
X^{\star}\text{ is unique, }\quad r^{\star}\eqdef\rank(X^{\star})\ll n,\label{eq:sdp-asm-p}
\end{equation}
and that the dual problem (\ref{eq:sdd}) admits a \emph{strictly
complementary} \emph{solution} (not necessarily unique)
\begin{equation}
\text{exists }S^{\star}\text{ such that }\rank(S^{\star})=n-r^{\star}.\label{eq:sdp-asm-d}
\end{equation}
\end{subequations}Here, $\S^{n}$ denotes the space of $n\times n$
real symmetric matrices with the inner product $\langle A,B\rangle=\tr(AB)$,
and $A\succeq B$ signifies that $A-B$ is positive semidefinite.
We assume without loss of generality that the constraints represented
by $\AA$ are linearly independent, meaning $\AA^{T}(y)=0$ implies
$y=0$. This assumption restricts the number of constraints to $m\le\frac{1}{2}n(n+1)$.
Both of these conditions are mild and are widely regarded as standard
assumptions. 

Currently, the low-rank SDP (\ref{eq:sdp-pd}) is most commonly approached
using the nonconvex Burer--Monteiro factorization~\cite{burer2003nonlinear},
which is to factorize $X=UU^{T}$ into an $n\times r$ low-rank factor
$U$, where $r\ge r^{\star}$ is a search rank parameter, and then
to locally optimize over $U$. While this can greatly reduce the number
of optimization variables, from $O(n^{2})$ down to as low as $O(n)$,
the loss of convexity can create significant convergence difficulties.
A basic but fundamental issue is the possibility of failure by getting
stuck at a spurious local minimum over the low-rank factor $U$~\cite{boumal2020deterministic}.
Even ignoring this potential failure mode, it can still take a second-order
method like cubic regularization or a trust-region method up to $O(\epsilon^{-3/2})$
iterations to reach $\epsilon$ second-order optimality over $U$,
where a primal-solution $(X,y,S)$ with $\epsilon$ duality gap can
(hopefully) be extracted~\cite{boumal2020deterministic}. In practice,
the Burer--Monteiro approach works well for some problems~\cite{boumal2016nonconvex,rosen2019se,chiu2023tight},
but as shown in Fig.~\ref{fig:main}, can be slow and unreliable
for other problems, particularly those with underlying nonsmoothness.

\begin{figure}[!t]
    \center 
    \includegraphics[width=1\textwidth]{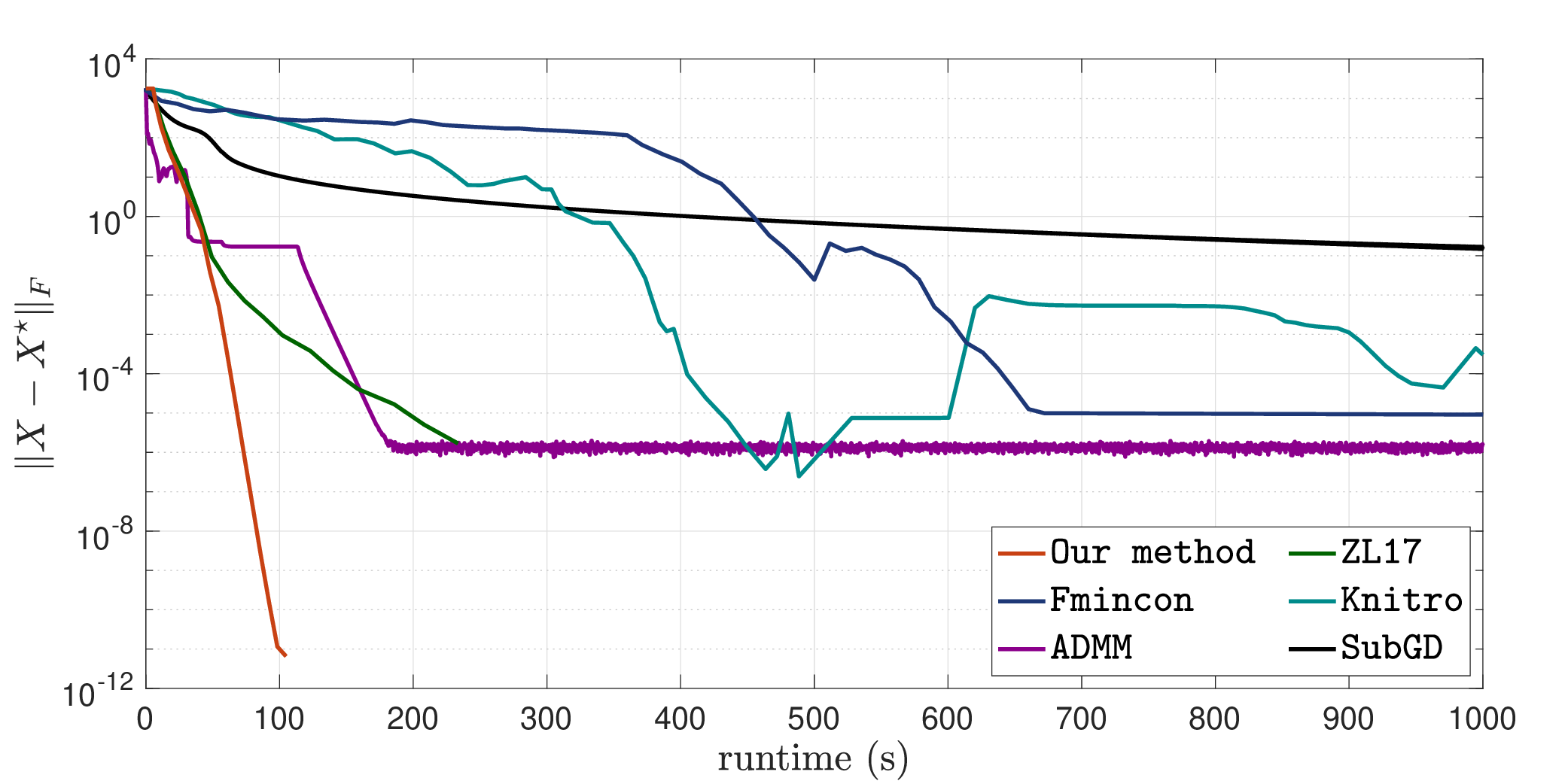}
    \caption{Error $\|X-X^{\star}\|_{F}$ against runtime for solving trace-regularized linear regression as an instance of (\ref{eq:sdp-pd}) using our proposed method (\texttt{Our method}); spectrally preconditioned interior-point method by Zhang and Lavaei~\cite{zhang2017modified} (\texttt{ZL17}); the Burer-Monteiro method, implemented with nonlinear programming solvers \texttt{Fmincon} \cite{matlabOTB} and \texttt{Knitro}~\cite{byrd2006k}, and equipped with analytic Hessians; alternating direction method of multipliers (\texttt{ADMM}); and subgradient method (\texttt{SubGD}). (See Section~\ref{sec:comparison} for details.)}\label{fig:main} 
\end{figure}

In this paper, we revisit classical primal-dual interior-point methods
(IPMs), because they maintain the convexity of the SDP (\ref{eq:sdp-pd}),
and therefore enjoy simple and rapid global convergence to $\epsilon$
duality gap in $O(\sqrt{n}\log(1/\epsilon))$ iterations. IPMs are
second-order methods that solve a sequence of logarithmic penalized
SDPs of the form 
\begin{align*}
X(\mu) & =\arg\min_{X\in\S^{n}}\{\inner CX-\mu\log\det(X):\AA(X)=b\},\\
y(\mu),S(\mu) & =\arg\max_{S\in\S^{n},y\in\R^{m}}\{\inner by+\mu\log\det(S):\AA^{T}(y)+S=C\},
\end{align*}
while progressively decreasing the duality gap parameter $\mu>0$
after each iteration. Using an IPM, the cost of solving (\ref{eq:sdp-pd})
is almost completely determined by the cost of solving the Newton
subproblem at each iteration\begin{subequations}\label{eq:newt}
\begin{gather}
\min_{\Delta X\in\S^{n}}\left\{ \inner{\tilde{C}}{\Delta X}+\frac{1}{2}\|W^{-1/2}(\Delta X)W^{-1/2}\|_{F}^{2}:\AA(\Delta X)=\tilde{b}\right\} ,\\
\max_{\Delta S\in\S^{n},\Delta y\in\R^{m}}\left\{ \inner{\tilde{b}}{\Delta y}-\frac{1}{2}\|W^{1/2}(\Delta S)W^{1/2}\|_{F}^{2}:\AA^{T}(\Delta y)+\Delta S=\tilde{C}\right\} ,
\end{gather}
in which the scaling point $W\in\S_{++}^{n}$ and the residuals $\tilde{b}\in\R^{m},\tilde{C}\in\S^{n}$
are iteration- and algorithm-dependent. In fact, using a short-step
strategy (with fixed step-sizes)~\cite[Theorem~6.1]{nesterov1998primal},
solving the SDP (\ref{eq:sdp-pd}) to $\epsilon$ duality gap costs
exactly the same as solving an instances
of (\ref{eq:newt}) to $\epsilon$ accuracy. By paying some small
per-iteration overhead, long-step strategies (with adaptive step-sizes)
can consistently reduce the iteration count down to 50-100 in practice~\cite{vandenberghe1996semidefinite,sturm2002implementation}. 

Instead, the traditional difficulty faced by IPMs is the high cost
of accurately solving the Newton subproblem (\ref{eq:newt}) at each
iteration. This difficulty arises due to need to use a scaling point
$W\in\S_{++}^{n}$ that is both \emph{dense} and becomes increasing\emph{
ill-conditioned} with decreasing values of $\mu$. Although our technique
is more broadly applicable, we will focus this paper on the primal-dual
Nesterov--Todd scaling~\cite{nesterov1997self,nesterov1998primal},
which is computed from the current iterates $X,S$ as follows 
\begin{equation}
W=X^{1/2}(X^{1/2}SX^{1/2})^{-1/2}X^{1/2}=[S^{1/2}(S^{1/2}XS^{1/2})^{-1/2}S^{1/2}]^{-1},
\end{equation}
\end{subequations}and appears in popular general-purpose IPM solvers
like SeDuMi~\cite{sturm1999sedumi}, SDPT3~\cite{toh1999sdpt3},
and MOSEK~\cite{mosek2015}. The ill-conditioning in $W$ makes it
increasingly difficult to solve (\ref{eq:newt}) to sufficiently high
accuracy for the outer IPM to continue making progress. Therefore,
most IPM solvers, like those listed above, solve (\ref{eq:newt})
by directly forming and factorizing its Karush--Kuhn--Tucker equations.
Unfortunately, the density of $W$ renders these equations dense,
irrespective of any sparsity in the data $C$ and $\AA$. For SDPs
with $m=\Theta(n^{2})$ constraints, the cost of solving (\ref{eq:newt})
by the direct approach is $\Theta(n^{6})$ time and $\Theta(n^{4})$
memory, hence limiting $n$ to no more than about 200. 

In this paper, we instead solve (\ref{eq:newt}) using a preconditioned
iterative approach, as a set of inner iterations within the outer
IPM iterations. Our main contribution is to use the low-rank structure
in (\ref{eq:sdp-pd}) to provably reduce the cost of computing a \emph{numerically
exact} solution to (\ref{eq:newt}) to $O(n^{3}r^{3}\log(1/\macheps)))$
time and $O(n^{2}r^{2})$ memory, where $r\ge r^{\star}$ is a search
rank parameter as before, and $\macheps$ is the machine precision.
Critically, these complexity figures hold even for SDPs with $m=\Theta(n^{2})$
constraints. This matches the per-iteration costs of second-order
implementations of the Burer--Monteiro method, but we additionally
retain the $O(\sqrt{n}\log(1/\epsilon))$ iteration bound enjoyed
by classical IPMs. As we now explain, our main challenge is being
able to maintain these complexity figures even at high accuracies,
where $\mu$ becomes very small, and $W$ becomes extremely ill-conditioned.

\subsection{Prior work: Spectral preconditioning}

By taking advantage of fast matrix-vector products with the data $\AA$,
an iterative approach can reduce the cost of solving the Newton subproblem
(\ref{eq:newt}) to as low as $O(n^{3})$ time and $O(n^{2})$ memory.
However, without an effective preconditioner, these improved complexity
figures are limited to coarse accuracies of $\mu\approx10^{-2}$.
The key issue is that the scaling point $W$ becomes ill-conditioned
as $\cond(W)=\Theta(1/\mu)$. Iterating until the numerical floor
yields a \emph{numerically exact} solution to (\ref{eq:newt}) that
is indistinguishable from the true solution under finite precision,
but the number of iterations to achieve this grows as $O(1/\mu\log(1/\macheps))$.
Truncating after a much smaller number of iterations yields an \emph{inexact}
solution to (\ref{eq:newt}) that can slow the convergence of the
outer IPM. 

Toh and Kojima \cite{toh2002solving} proposed the first effective
preconditioner for SDPs. Their critical insight is the empirical observation
that the scaling point $W\in\S_{++}^{n}$ becomes ill-conditioned
as $\cond(W)=\Theta(1/\mu)$ only because its eigenvalues split into
two well-conditioned clusters:
\begin{equation}
W=\underbrace{Q\Lambda Q^{T}}_{\text{top }r^{\star}\text{ eigenvalues}}+\underbrace{Q_{\perp}\Lambda_{\perp}Q_{\perp}^{T},}_{\text{bottom }n-r^{\star}\text{ eigenvalues}}\,\text{where }\Lambda=\Theta(1/\sqrt{\mu})\text{ and }\Lambda_{\perp}=\Theta(\sqrt{\mu}).\tag{A0}\label{eq:A0}
\end{equation}
Based on this insight, they formulated a \emph{spectral preconditioner}
that attempts to counteract the ill-conditioning between the two clusters,
and demonstrated promising computational results. However, the empirical
effectiveness was not rigorously guaranteed, and its expensive setup
cost of $O(mn^{3})$ time made its overall solution time comparable
to a direct solution of (\ref{eq:newt}), particularly for problems
with $m=\Theta(n^{2})$ constraints. 

Inspired by the above, Zhang and Lavaei~\cite{zhang2017modified}
proposed a spectral preconditioner that is much cheaper to set up,
and also rigorously guarantees a high quality of preconditioning under
the splitting assumption (\ref{eq:A0}). Their key idea is to rewrite
the scaling point as the low-rank perturbation of a well-conditioned
matrix, and then to approximate the well-conditioned part by the identity
matrix
\[
W=\quad\underbrace{Q(\Lambda-\tau)Q^{T}}_{\text{low-rank}}\quad+\quad\tau\cdot\underbrace{(QQ^{T}+\tau^{-1}Q_{\perp}\Lambda_{\perp}Q_{\perp}^{T})}_{\text{well-conditioned}}\quad\approx\quad UU^{T}\quad+\quad\tau I.
\]
Indeed, it is easy to verify that the joint condition number between
$W$ and $UU^{T}+\tau I$ is bounded under (\ref{eq:A0}) even as
$\mu\to0^{+}$. Substituting $W\approx UU^{T}+\tau I$ into (\ref{eq:newt})
reveals a low-rank perturbed problem that admits a closed-form solution
via the Sherman--Morrison--Woodbury formula, which they proposed
to use as a preconditioner. Assuming (\ref{eq:A0}) and the availability
of fast matrix-vector products that we outline as (\ref{eq:A3}) below,
they proved that their method computes a numerically exact solution
in $O(n^{3}r^{3}+n^{3}r^{2}\log(1/\macheps))$ time and $O(n^{2}r^{2})$
memory. This same preconditioning idea was later adopted in the solver
\texttt{Loraine}~\cite{habibi2023loraine}.

In practice, however, the Zhang--Lavaei preconditioner loses its
effectiveness for smaller values of $\mu$, until abruptly failing
at around $\mu\approx10^{-6}$. This is an inherent limitation of
preconditioning in finite-precision arithmetic; to precondition a
matrix with a diverging condition number $\Theta(1/\mu^{2})$, the
preconditioner's condition number must also diverge as $\Theta(1/\mu^{2})$.
As $\mu$ approaches zero, it eventually becomes impossible to solve
the preconditioner accurately enough for it to remain effective. This
issue is further exacerbated by the use of the Sherman--Morrison--Woodbury
formula, which is itself notorious for numerical instability. 

Moreover, the effectiveness of both prior preconditioners critically
hinges on the splitting assumption (\ref{eq:A0}), which lacks a rigorous
justification. It was informally argued in \cite{toh2002solving,toh2004solving}
that the assumpution holds under strict complementarity (\ref{eq:sdp-asm-d})
and the \emph{primal-dual nondegeneracy} of Alizadeh, Haeberly, and
Overton~\cite{alizadeh1997complementarity}. But primal-dual nondegeneracy
is an excessively strong assumption, as it would require the number
of constraints $m$ to be within the range of
\[
\frac{1}{2}r^{\star}(r^{\star}+1)\le m\le nr^{\star}-\frac{1}{2}r^{\star}(r^{\star}-1).
\]
In particular, it is not applicable to the diverse range of low-rank
matrix recovery problems, like matrix sensing~\cite{recht2010guaranteed,candes2011tight},
matrix completion~\cite{candes2009exact}, phase retrieval~\cite{candes2013phaselift},
that necessarily require $m>nr^{\star}-\frac{1}{2}r^{\star}(r^{\star}-1)$
measurements in order to uniquely recover an underlying rank-$r^{\star}$
ground truth. It is also not applicable to problems with $m=\Theta(n^{2})$
constraints, which as we explained above, are the most difficult SDPs
to solve for a fixed value of $n$. 

\subsection{This work: Well-conditioned reformulation and preconditioner}

In this paper, we derive a \emph{well-conditioned reformulation} of
the Newton subproblem (\ref{eq:newt}) that costs just $O(n^{3})$
time and $O(n^{2})$ memory to set up. In principle, as the reformulation
is already well-conditioned by construction, it allows any iterative
method to compute a numerically exact solution in $O(\log(1/\macheps))$
iterations for all values of $0<\mu\le1$. In practice, the convergence
rate can be substantially improved by the use of a \emph{well-conditioned
preconditioner}. Assuming (\ref{eq:A0}) and the fast matrix-vector
products in (\ref{eq:A3}) below, we provably compute a numerically
exact solution in $O(n^{3}r^{3}+n^{3}r^{2}\log(1/\macheps))$ time
and $O(n^{2}r^{2})$ memory. But unlike Zhang and Lavaei~\cite{zhang2017modified},
our well-conditioned preconditioner maintains its full effectiveness
even at extremely high accuracies like $\mu\approx10^{-12}$, hence
fully addressing a critical weakness of this prior work. 

Our main theoretical results are a tight upper bound on the condition number
of the reformulation and the preconditioner (\thmref{main}), and
a tight bound on the number of preconditioned iterations needed to
solve (\ref{eq:newt}) to $\epsilon$ accuracy (\corref{pcg}). These
rely on the same splitting assumption (\ref{eq:A0}) as prior work,
which we show is implied by a strengthened version of strict complementarity
(\ref{eq:sdp-asm-d}).  In \lemref{eigsplit}, we prove that (\ref{eq:A0})
holds if the primal-dual iterates $(X,S)$ used to compute the scaling
point $W$ in (\ref{eq:newt}) satisfy the following centering condition
(the spectral norm $\|\cdot\|$ is defined as the maximum singular
value):
\begin{equation}
\left\Vert \frac{1}{\mu}X^{1/2}SX^{1/2}-I\right\Vert <1,\qquad\|X-X^{\star}\|=O(\mu),\qquad\|S-S^{\star}\|=O(1).\tag{A1}\label{eq:A1}
\end{equation}
In particular, the restrictive primal-dual nondegeneracy assumption
adopted by prior work \cite{toh2002solving,toh2004solving} is completely
unnecessary. By repeating the algebraic derivations in \cite{luo1998superlinear},
it can be shown that (\ref{eq:A1}) holds under strict complementarity
(\ref{eq:sdp-asm-d}) if the IPM maintains its primal-dual iterates
$(X,S)$ within the ``wide'' neighborhood $\N_{\infty}$ \cite{wright1997primal,nesterov1997self,benson2000solving}.
In turn, most IPMs that achieve an $\epsilon$ duality gap in $O(\sqrt{n}\log(1/\epsilon))$
iterations---including essentially all theoretical short-step methods
\cite{nesterov1997self,nesterov1998primal,sturm1999symmetric} but
also many practical methods like the popular solver SeDuMi \cite{sturm1999sedumi}---do
so by keeping their iterates within the $\N_{\infty}$ neighborhood
\cite{sturm2002implementation}. In Section~\ref{sec:verify_thm},
we experimentally verify for SDPs satisfying strict complementarity
(\ref{eq:sdp-asm-d}) that SeDuMi~\cite{sturm1999sedumi} keeps its
primal-dual iterates $(X,S)$ sufficiently centered to satisfy (\ref{eq:A1}).

Under (\ref{eq:A1}), which implies the splitting assumption (\ref{eq:A0})
adopted in prior work via \lemref{eigsplit}, we prove that $O(\log(1/\epsilon))$
preconditioned iterations solve (\ref{eq:newt}) to $\epsilon$ accuracy
for all values of $0<\mu\le1$. In order for the reformulation and
the preconditioner to remain well-conditioned, however, we require
a strengthened version of the primal uniqueness assumption (\ref{eq:sdp-asm-p}).
Concretely, we require the linear operator $\AA$ to be injective
with respect to the tangent space of the manifold of rank-$r^{\star}$
positive semidefinite matrices, evaluated at the unique primal solution
$X^{\star}=U^{\star}U^{\star T}$:
\begin{gather}
\AA(U^{\star}V^{T}+VU^{\star T})=0\quad\iff\quad U^{\star}V^{T}+VU^{\star T}=0.\tag{A2}\label{eq:A2}
\end{gather}
This is stronger than primal uniqueness because it is possible for
a rank-$r^{\star}$ primal solution to be unique with just $m=\frac{1}{2}r^{\star}(r^{\star}+1)$
constraints~\cite{alizadeh1997complementarity}, but (\ref{eq:A2})
can hold only when the number of constraints $m$ is no less than
the dimension of the tangent space 
\[
m\ge\frac{1}{2}r^{\star}(r^{\star}+1)+r^{\star}(n-r^{\star})=nr^{\star}-\frac{1}{2}r^{\star}(r^{\star}-1).
\]
Fortunately, many rigorous proofs of primal uniqueness, such as for
matrix sensing~\cite{recht2010guaranteed,candes2011tight}, matrix
completion~\cite{candes2009exact}, phase retrieval~\cite{candes2013phaselift},
actually work by establishing (\ref{eq:A2}). In problems where (\ref{eq:A2})
does not hold, it becomes possible for our reformulation and preconditioner
to grow ill-conditioned in the limit $\mu\to0^{+}$. For these cases,
the behavior of our method matches that of Zhang and Lavaei~\cite{zhang2017modified};
down to modest accuracies of $\mu\approx10^{-6}$, our preconditioned
iterations will still rapidly converge to $\epsilon$ accuracy in
$O(\log(1/\epsilon))$ iterations.

Our time complexity figures require the same fast matrix-vector products
as previously adopted by Zhang and Lavaei~\cite{zhang2017modified}
(and implicitly adopted by Toh and Kojima~\cite{toh2002solving}).
We assume, for $U,V\in\R^{n\times r}$ and $y\in\R^{m}$, that the
two matrix-vector products 
\begin{equation}
\AA(UV^{T}+VU^{T})=\left[\inner{A_{i}}{UV^{T}+VU^{T}}\right]_{i=1}^{m},\quad\AA^{T}(y)U=\sum_{i=1}^{m}y_{i}A_{i}U\tag{A3}\label{eq:A3}
\end{equation}
can both be performed in $O(n^{2}r)$ time and $O(n^{2})$ storage.
This is satisfied in problems with suitable sparsity
and/or Kronecker structure. In Section~\ref{sec:verify_thm}, we
rigorously prove that the robust matrix completion and sensor network
localization problems satisfy (\ref{eq:A3}) due to suitable sparsity
in the operator $\AA$.

\subsection{Related work}

Our well-conditioned reformulation of the Newton subproblem (\ref{eq:newt})
is inspired by Toh~\cite{toh2004solving} and shares superficial
similarities. Indeed, the relationship between our reformulation and
the Zhang--Lavaei preconditioner~\cite{zhang2017modified} has an
analogous relationship to Toh\textquoteright s reformulation with
respect to the Toh--Kojima preconditioner~\cite{toh2002solving}.
However, the actual methods are very different. First, Toh's reformulation
remains well-conditioned only under primal-dual nondegeneracy, which
as we explained above, makes it inapplicable to most low-rank matrix
recovery problems, as well as problems with $m=\Theta(n^{2})$ constraints.
Second, the effectiveness of Toh's preconditioner is not rigorously
guaranteed; due to its reliance on the PSQMR method, the Faber--Manteuffel
Theorem~\cite{faber1984necessary} says that it is fundamentally
impossible to prove convergence for the preconditioned iterative method.
Third, our set-up cost of $O(n^{3}r^{3})$ is much cheaper than Toh's
set-up cost of $O(mn^{3})$; this parallels the reduction in set-up
cost achieved by the Zhang--Lavaei preconditioner~\cite{zhang2017modified}
over the Toh--Kojima preconditioner~\cite{toh2002solving}.

We mention another important line of work on the use of partial Cholesky
preconditioners \cite{bergamaschi2004preconditioning,gondzio2012matrix,bellavia2013matrix},
which was extended to IPMs for SDPs in~\cite{bellavia2019inexact}.
These preconditioners are based on the insight that a small number
of ill-conditioned dimensions in the Newton subproblem (\ref{eq:newt})
can be corrected by a low-rank partial Cholesky factorization of its
underlying Schur complement. However, to the best of our knowledge,
such preconditioners do not come with rigorous guarantees of an improvement. 

Outside of SDPs, our reformulation is also inspired by the layered
least-squares of Bobrovnikova and Vavasis~\cite{bobrovnikova2001accurate};
indeed, our method can be viewed as the natural SDP generalization.
However, an important distinction is that they apply MINRES directly
to solve the indefinite reformulation, whereas we combine PCG with
an indefinite preconditioner to dramatically improve the convergence
rate. This idea of using PCG to solve indefinite equations with an
indefinite preconditioner had first appeared in the scientific computing
literature~\cite{lukvsan1998indefinitely,gould2000iterative,rozloznik2002krylov}.

We mention that first-order methods have also been derived to exploit
the low-rank structure of the SDP (\ref{eq:sdp-pd}) without giving
up the convexity~\cite{cai2010singular,yurtsever2021scalable,yang2023inexact}.
These methods gain their efficiency by keeping all their iterates
low-rank, and by performing convex updates on the low-rank iterates
$X=UU^{T}$ while maintaining them in low-rank factored form $X_{+}=U_{+}U_{+}^{T}$.
While convex first-order methods cannot become stuck at a spurious
local minimum, they still require $\mathrm{poly}(1/\epsilon)$ iterations
to converge to an $\epsilon$-accurate solution of (\ref{eq:sdp-pd}).
In contrast, our proposed method achieves the same in just $O(\log(1/\epsilon))$
iterations, for an exponential factor improvement.

Finally, for low-rank SDPs with small treewidth sparsity, chordal conversion methods \cite{fukuda2001exploiting,kim2011exploiting}
compute an $\epsilon$ accurate solution in guaranteed $O(n^{1.5}\log(1/\epsilon))$
time and $O(n)$ memory~\cite{zhang2021sparse,zhang2023parameterized}.
Where the property holds, chordal conversion methods achieve state-of-the-art
speed and reliability. Unfortunately, many real-world problems do
not enjoy small treewidth sparsity~\cite{maniu2019experimental}.
For low-rank SDPs that are sparse but non-chordal, our proposed method
can solve them in as little as $O(n^{3.5}\log(1/\epsilon))$ time
and $O(n^{2})$ memory.

\subsection*{Notations}
Write $\R^{n\times n}\supseteq\S^{n}\supseteq\S_{+}^{n}\supseteq\S_{++}^{n}$ as the sets of $n\times n$ real matrices, real symmetric
matrices, positive semidefinite matrices, and positive definite matrices, respectively.
Write $\vector:\R^{n\times n}\to\R^{n^{2}}$ as the column-stacking
vectorization, and $A\otimes B$ as the Kronecker product satisfying $(B\otimes A)\vector(X)=\vector(AXB^{T})$. Write $\|\cdot\|$ and $\|\cdot\|_F$ as spectral norm and Frobenius norm, respectively. Let $\Psi_{n}$ denote the $n^{2}\times n(n+1)/2$ basis for the set of vectorized real symmetric matrices $\vector(\S^{n})\subseteq\vector(\R^{n\times n})$. We define the \emph{symmetric vectorization} of $X$ as 
\[
\svec(X)\eqdef[X_{11},\sqrt{2}X_{21},\dots,\sqrt{2}X_{n1},X_{22},\sqrt{2}X_{32},\dots,\sqrt{2}X_{n2},\dots,X_{nn}]^T.
\]
We note that $\inner XS=\inner{\svec(X)}{\svec(S)}=\inner{\vector(X)}{\vector(S)}$ for all $X,S\in\S^{n}$.
Like the identity matrix, we will frequently suppress the subscript
$n$ if the dimensions can be inferred from context. Given $A,B\in\R^{n\times r}$, we define their \emph{symmetric Kronecker product} as $A\skron B\eqdef\frac{1}{2}\Psi_{n}^{T}(A\otimes B+B\otimes A)\Psi_{r}$
so that $(A\skron B)\svec(X)=(B\skron A)\svec(X)=\svec(AXB^{T})$ holds for all $X\in\S^{r}$.
(See also~\cite{alizadeh1997complementarity,toh2002solving,toh2004solving}
for earlier references on the symmetric vectorization and Kronecker
product.) 

\section{Background: Iterative solvers for symmetric indefinite equations}

\subsection{MINRES}

Given a system of equations $Ax=b$ with a symmetric $A$ and right-hand
side $b$, we define the corresponding minimum residual (MINRES) iterations
with respect to positive definite preconditioner $P$, and initial
point $x_{0}$ implicitly in terms of its Krylov optimality condition. 
\begin{definition}[Minimum residual]
Given $A\in\S^{n}$ and $P\in\S_{++}^{n}$ and $b,x_{0}\in\R^{n}$,
we define $\MINRES_{k}(A,b,P,x_{0})\eqdef P^{1/2}y_{k}$ where
\[
y_{k}=\arg{\textstyle \min_{y}}\{\|\tilde{A}y-\tilde{b}\|:y\in y_{0}+\mathrm{span}\{\tilde{b},\tilde{A}\tilde{b},\dots,\tilde{A}^{k-1}\tilde{b}\}\}
\]
and $\tilde{A}=P^{-1/2}AP^{-1/2}$ and $\tilde{b}=P^{-1/2}b$ and
$y_{0}=P^{-1/2}x_{0}$. 
\end{definition}

The precise implementation details of MINRES are complicated, and
can be found in the standard reference~\cite[Algorithm 4]{greenbaum1997iterative}.
We only note that the method uses $O(n)$ memory over all iterations,
and that each iteration costs $O(n)$ time plus the the matrix-vector
product $p\mapsto Ap$ and the preconditioner linear solve $r\mapsto P^{-1}r$.
In exact arithmetic, MINRES terminates with the exact solution in
$n$ iterations. With round-off noise, however, exact termination
does not occur. Instead, the behavior of the iterates are better predicted
by the following bound.
\begin{proposition}[Symmetric indefinite]
\label{prop:minres}Given $A\in\S^{n}$ and $P\in\S_{++}^{n}$ and
$b,x_{0}\in\R^{n}$, let $x_{k}=\MINRES_{k}(A,b,P,x_{0}).$ Then,
we have $\|Ax_{k}-b\|\le\epsilon$ in at most
\[
k\le\left\lceil \frac{1}{2}\kappa\log\left(\frac{2\|Ax^{0}-b\|}{\epsilon}\right)\right\rceil \text{ iterations}
\]
where $\kappa=\cond(P^{-1/2}AP^{-1/2})$. 
\end{proposition}

This paper proposes a well-conditioned reformulation $Ax=b$ for the
Newton subproblem (\ref{eq:newt}), whose condition number $\kappa=\cond(A)$
remains bounded over all iterations of the IPM. In principle, \propref{minres}
says that it takes $O(\log(1/\epsilon))$ iterations to solve this
reformulation to $\epsilon$ accuracy.

\subsection{PCG}

Given a system of equations $Ax=b$ with a symmetric $A$ and right-hand
side $b$, we define the corresponding preconditioned conjugate gradient
(PCG) iterations with respect to preconditioner $P$ not necessarily
positive definite, and initial point $x_{0}$ explicitly in terms
of its iterations.
\begin{definition}[Preconditioned conjugate gradients]
\label{def:pcg}\sloppy Given $A,P\in\S^{n}$ and $b,x_{0}\in\R^{n}$,
we define $\PCG_{k}(A,b,P,x_{0})\eqdef x_{k}$ where for $j\in\{0,1,2,\dots,k\}$
\begin{gather*}
x_{j+1}=x_{j}+\alpha_{j}p_{j},\quad Pz_{j+1}=r_{j+1}=r_{j}-\alpha_{j}Ap_{j},\quad p_{j+1}=z_{j+1}+\beta_{j}p_{j},\\
\alpha_{j}=\inner{r_{j}}{z_{j}}/\inner{p_{j}}{Ap_{j}},\qquad\beta_{j}=\inner{r_{j+1}}{z_{j}}/\inner{r_{j}}{z_{j}},
\end{gather*}
and $Pz_{0}=r_{0}=b-Ax_{0}$ and $p_{0}=z_{0}$.
\end{definition}

We can verify from \defref{pcg} that the method uses $O(n)$ memory
over all iterations, and that each iteration costs $O(n)$ time plus
the the matrix-vector product $p\mapsto Ap$ and the preconditioner
linear solve $r\mapsto P^{-1}r$. The following is a classical iteration
bound for when PCG is used to solve a symmetric positive definite
system of equations with a positive definite preconditioner. 
\begin{proposition}[Symmetric positive definite]
\label{prop:pcg}Given $A,P\in\S_{++}^{n}$ and $b,x_{0}\in\R^{n}$,
let $x_{k}=\PCG_{k}(A,b,P,x_{0})$. Then, both $\|Ax_{k}-b\|\le\epsilon$
holds in at most
\[
k\le\left\lceil \frac{1}{2}\sqrt{\kappa}\log\left(\frac{2\sqrt{\kappa}\|Ax_{0}-b\|}{\epsilon}\right)\right\rceil \text{ iterations}
\]
where $\kappa=\cond(P^{-1/2}AP^{-1/2})$. 
\end{proposition}

In the late 1990s and early 2000s, it was pointed out in several parallel
works~\cite{lukvsan1998indefinitely,gould2000iterative,rozloznik2002krylov}
that PCG can also be used to solve \emph{indefinite} systems, with
the help of an indefinite preconditioner. The proof of \propref{indefpcg}
follows immediately by substituting into \defref{pcg}.
\begin{proposition}[Indefinite preconditioning]
\label{prop:indefpcg}Given $A,P\in\S_{++}^{n}$ and $C\in\S_{++}^{r}$
and $B\in\R^{n\times r}$ and $x_{0},f\in\R^{n}$ and $g\in\R^{r}$,
let\begin{subequations}
\begin{gather}
\begin{bmatrix}u_{k}\\
v_{k}
\end{bmatrix}=\mathrm{PCG}_{k}\left(\begin{bmatrix}A & B\\
B^{T} & -C
\end{bmatrix},\begin{bmatrix}f\\
g
\end{bmatrix},\begin{bmatrix}P & B\\
B^{T} & -C
\end{bmatrix},\begin{bmatrix}x_{0}\\
C^{-1}(B^{T}x_{0}-g)
\end{bmatrix}\right),\label{eq:indef1}\\
x_{k}=\mathrm{PCG}_{k}(A+BC^{-1}B^{T},f+BC^{-1}g,P+BC^{-1}B^{T},x_{0}).\label{eq:indef2}
\end{gather}
\end{subequations}In exact arithmetic, we have $u_{k}=x_{k}$ and
$v_{k}=C^{-1}(B^{T}x_{k}-g)$ for all $k\ge0$.
\end{proposition}

\begin{proof}
Let $\alpha_{k},p_{k}$ and $\tilde{\alpha}_{k},\tilde{p}_{k}$ denote
the internal iterates generated by PCG on (\ref{eq:indef1}) and (\ref{eq:indef2}),
respectively. Assuming $u_{k}=x_{k}$ and $v_{k}=C^{-1}(B^{T}x_{k}-g)$,
one can verify from \defref{pcg} that $\alpha_{k}=\tilde{\alpha}_{k}$, $p_{k}=\left[\begin{smallmatrix}I\\C^{-1}B^{T} &\ C^{-1}\end{smallmatrix}\right]
\left[\begin{smallmatrix}\tilde{p}_{k}\\0^{\vphantom{1}}\end{smallmatrix}\right]$
and
\begin{align*}
\begin{bmatrix}u_{k+1}\\
v_{k+1}
\end{bmatrix}=\begin{bmatrix}I\\
C^{-1}B^{T} &\ C^{-1}
\end{bmatrix}\begin{bmatrix}x_{k}+\tilde{\alpha}_{k}\tilde{p}_{k}\\
-g
\end{bmatrix}=\begin{bmatrix}x_{k+1}\\
C^{-1}(B^{T}x_{k+1}-g)
\end{bmatrix}.
\end{align*}
With $u_{0}=x_{0}$ and $v_{0}=C^{-1}(B^{T}x_{0}-g)$,
the desired result follows by induction. 
\end{proof}
Therefore, the indefinite preconditioned PCG (\ref{eq:indef1}) is
guaranteed to converge because it is mathematically equivalent to
PCG on the underlying positive definite Schur complement system (\ref{eq:indef2}).
Nevertheless, (\ref{eq:indef1}) is more preferable when the matrix
$C$ is close to singular, because the two indefinite matrices in
(\ref{eq:indef1}) can remain well-conditioned even as $\cond(C)\to\infty$.
As the two Schur complements become increasingly ill-conditioned,
the preconditioning effect in (\ref{eq:indef2}) begins to fail, but
the indefinite PCG in (\ref{eq:indef1}) will maintain its rapid convergence
as if perfectly conditioned. 

\section{Proposed method and summary of results}
 
Given scaling point $W\in\S_{++}^{n}$
and residuals $\tilde{b}\in\R^{m},\tilde{C}\in\S^{n}$, our goal is
to compute $\Delta X\in\S^{n}$ and $\Delta y\in\R^{m}$ in closed-form
via the Karush--Kuhn--Tucker equations
\begin{equation}
\left\Vert \begin{bmatrix}-(W\skron W)^{-1} & \A^{T}\\
\A & 0
\end{bmatrix}\begin{bmatrix}\svec(\Delta X)\\
\Delta y
\end{bmatrix}-\begin{bmatrix}\svec(\tilde{C})\\
\tilde{b}
\end{bmatrix}\right\Vert \quad\le\quad\epsilon\tag{Newt-\ensuremath{\epsilon}}\label{eq:Newt-eps}
\end{equation}
where $\A=[\svec(A_{1}),\dots,\svec(A_{m})]^T$, and then recover $\Delta S=\tilde{C}-\AA^{T}(\Delta y)$.
We focus exclusively on achieving a small enough $\epsilon$
as to be considered numerically exact. 

As mentioned in the introduction, the main difficulty is that the
scaling point $W$ becomes progressively ill-conditioned as the accuracy
parameter is taken $\mu\to0^{+}$. The following is a more precise
statement of (\ref{eq:A1}). 

\begin{assume}\label{asm:analy}There exist absolute constants $L\ge0$
and $0\le\delta<1$ such that the iterates $X,S\in\S_{++}^{n}$, where $W=X^{1/2}(X^{1/2}SX^{1/2})^{-1/2}X^{1/2}$, satisfy 
\[
\left\Vert \frac{1}{\mu}X^{1/2}SX^{1/2}-I\right\Vert \le\delta,\qquad\|X-X^{\star}\|\le L\mu,\qquad\|S-S^{\star}\|\le L,
\]
with respect to $X^{\star},S^{\star}\in\S_{+}^{n}$ that satisfy $X^{\star}S^{\star}=0$
and $\chi_{1}^{-1}\cdot I\preceq X^{\star}+S^{\star}\preceq I$.\end{assume}

Under \asmref{analy}, the eigenvalues of $W$ split into a size-$r^{\star}$
cluster that diverges as $\Theta(1/\sqrt{\mu})$, and a size-$(n-r^{\star})$
cluster that converges to zero as $\Theta(\sqrt{\mu})$. Therefore,
its condition number scales as $\cond(W)=\Theta(1/\mu)$, and this
causes the entire system to become ill-conditioned with a condition
number like $\Theta(1/\mu^{2})$.
\begin{lemma}
\label{lem:eigsplit}Under \asmref{analy}, let $r=\rank(X^{\star})$
and $0<\mu\le1$. The eigenvalues of $W$ satisfy
\[
\begin{array}{ccccccc}
C_{1}/\sqrt{\mu} & \ge & \lambda_{1}(W) & \ge & \lambda_{r}(W) & \ge & 1/(C_{2}\sqrt{\mu}),\\
C_{1}\sqrt{\mu} & \ge & \lambda_{r+1}(W) & \ge & \lambda_{n}(W) & \ge & \sqrt{\mu}/C_{1},
\end{array}
\]
where $C_{1}=\frac{1+L}{1-\delta}$, and $C_{2}=4\chi_{1}+\frac{2L^{2}\chi_{1}}{1-\delta}$.
\end{lemma}

\begin{proof}
Write $w_{i}\equiv\lambda_{i}(W)$ and $x_{i}^{\star}\equiv\lambda_{i}(X^{\star})$
and $s_{i}^{\star}\equiv\lambda_{i}(S^{\star})$, and let $\#$ denote
the matrix geometric mean operator of Ando~\cite{ando1979concavity}.
Substituting the matrix arithmetic-mean geometric-mean~\cite[Corollary 2.1]{ando1979concavity}
inequalities $\sqrt{\mu}W=X\#(\mu S^{-1})\preceq\frac{1}{2}(X+\mu S^{-1})$
and $\sqrt{\mu}W^{-1}=S\#(\mu X^{-1})\preceq\frac{1}{2}(S+\mu X^{-1})$
into $\|\frac{1}{\mu}X^{1/2}SX^{1/2}-I\|\le\delta$ yields 
\[
\frac{1}{1+\delta}\cdot X\preceq\sqrt{\mu}W\preceq\frac{1}{1-\delta}\cdot X,\quad\frac{1}{1+\delta}\cdot S\preceq\sqrt{\mu}W^{-1}\preceq\frac{1}{1-\delta}\cdot S,
\]
which we combine with $\|X-X^{\star}\|\le L\mu$ and $\|S-S^{\star}\|\le L$
to obtain
\begin{gather*}
\sqrt{\mu}w_{1}\le\frac{x_{1}^{\star}+L\mu}{1-\delta},\quad\sqrt{\mu}w_{n}^{-1}\le\frac{s_{1}^{\star}+L}{1-\delta},\quad\sqrt{\mu}w_{r+1}\le\frac{0+L\mu}{1-\delta},\\
\sqrt{\mu}w_{r}\ge\frac{x_{r}^{\star}-L\mu}{1+\delta},\quad\sqrt{\mu}w_{r}^{-1}\le\frac{0+L}{1-\delta}.
\end{gather*}
The first row yields $\sqrt{\mu}w_{1}\le C_{1}$ for $0<\mu\le1$,
and $\sqrt{\mu}w_{n}^{-1}\le C_{1}$ and $w_{r+1}\le C_{1}\sqrt{\mu}$.
We can select a \emph{fixed} point $\hat{\mu}>0$ and combine the
second row to yield 
\[
\sqrt{\mu}w_{r}\ge\max\left\{ \frac{x_{r}^{\star}-L\mu}{1+\delta},\frac{1-\delta}{L}\mu\right\} \ge\min\left\{ \frac{x_{r}^{\star}-L\hat{\mu}}{1+\delta},\frac{1-\delta}{L}\hat{\mu}\right\} ,
\]
which is valid for all $\mu>0$. In particular, if we choose $\hat{\mu}$
such that $x_{r}^{\star}-L\hat{\mu}=\frac{1}{2}x_{r}^{\star}$, then
$\frac{x_{r}^{\star}-L\mu}{1+\delta}\ge\frac{1}{4}x_{r}^{\star}=\frac{1}{4\chi_{1}}\ge\frac{1}{C_{2}}$
and $\frac{1-\delta}{L}\hat{\mu}=\frac{1-\delta}{L}(\frac{x_{r}^{\star}}{2L})\ge\frac{1-\delta}{2L^{2}\chi_{1}}\ge\frac{1}{C_{2}}$. 
\end{proof}
To derive a reformulation of (\ref{eq:Newt-eps}) that remains well-conditioned
even as $\mu\to0^{+}$, our approach is to construct a \emph{low-rank
plus well-conditioned} decomposition of the scaling matrix 
\[
W\skron W=\quad\underbrace{\Q\Sig^{-1}\Q^{T}}_{\text{low-rank}}\quad+\quad\underbrace{\tau^{2}\cdot\E}_{\text{well-conditioned}}.
\]
First, we choose a rank parameter $r\ge\rank(X^{\star})$, and then partition
the orthonormal eigendecomposition of the scaling matrix $W$ into
two parts\begin{subequations}\label{eq:EQSigdef}
\begin{equation}
W=\quad\underbrace{Q\Lambda Q^{T}}_{\text{top }r\text{ eigenvalues}}\quad+\quad\underbrace{Q_{\perp}\Lambda_{\perp}Q_{\perp}^{T}.}_{\text{bottom }n-r\text{ eigenvalues}}\label{eq:Wsplit}
\end{equation}
Then, we choose a threshold parameter $\tau=\frac{1}{2}\lambda_{r+1}(W)$
and define
\begin{gather}
\E=E\skron E,\qquad E=QQ^{T}+\tau^{-1}\cdot Q_{\perp}\Lambda_{\perp}Q_{\perp}^{T},\\
\Q=[Q\skron Q,\quad\sqrt{2}\Psi_{n}^{T}(Q\otimes Q_{\perp})],\\
\Sig=\diag(\Lambda\skron\Lambda-\tau^{2}I,\quad(\Lambda-\tau I)\otimes\Lambda_{\perp})^{-1}.
\end{gather}
\end{subequations}The statement below characterizes this decomposition. 
\begin{proposition}[Low-rank plus well-conditioned decomposition]
\label{prop:correct}Given $W\in\S_{++}^{n}$, choose $1\le r\le n$
and $0<\tau<\lambda_{r}(W)$, and define $\E,\Q,\Sig$ as in (\ref{eq:EQSigdef}).
Then, $W\skron W=\Q\Sig^{-1}\Q^{T}+\tau^{2}\cdot\E$ holds with $\Sig\succ0,$
and $\E\succ0$, and 
\[
\colspan(\Q)=\{\svec(VQ^{T}):V\in\R^{n\times r}\},\qquad\Q^{T}\Q=I_{d}
\]
where $d=nr-\frac{1}{2}r(r-1)$. Moreover, under \asmref{analy},
if $r\ge\rank(X^{\star})$ and $\tau=\frac{1}{2}\lambda_{r+1}(W)$,
then
\[
\begin{array}{ccccccc}
4 & \ge & \lambda_{\max}(\E) & \ge & \lambda_{\min}(\E) & \ge & 1/C_{1}^{4}\\
C_{1}^{4} & \ge & \tau^{2}\cdot\lambda_{\max}(\Sig) & \ge & \tau^{2}\cdot\lambda_{\min}(\Sig) & \ge & \mu^{2}/(4C_{1}^{4})
\end{array}
\]
for all $0<\mu\le1$, where $C_{1},C_{2}$ are as defined in \lemref{eigsplit}.
If additionally $r=\rank(X^{\star})$, then $\mu\cdot C_{1}^{3}C_{2}\ge\tau^{2}\cdot\lambda_{\max}(\Sig)$
for all $0<\mu\le1$. 
\end{proposition}

\begin{proof}
The result follows from straightforward linear algebra and by substituting
\lemref{eigsplit}. For completeness, we provide a proof in \appref{correctness}.
\end{proof}
We propose using an iterative Krylov subspace method to solve the
following \begin{subequations}\label{eq:augrec}
\begin{equation}
\left\Vert \begin{bmatrix}\A\E\A^{T} & \A\Q\\
\Q^{T}\A^{T} & -\tau^{2}\cdot\Sig
\end{bmatrix}\begin{bmatrix}u\\
v
\end{bmatrix}-\begin{bmatrix}b+\tau^{2}\A\E c\\
\tau^{2}\cdot\Q^{T}c
\end{bmatrix}\right\Vert \le\epsilon,\label{eq:augsys}
\end{equation}
and then recover a solution to (\ref{eq:Newt-eps}) via the following
\begin{equation}
\Delta y=\tau^{-2}\cdot u,\qquad\svec(\Delta X)=\E(\A^{T}u-\tau^{2}c)+\Q v.\label{eq:recover}
\end{equation}
\end{subequations}Our main result is a theoretical guarantee that
our specific choice of $\E,\Q,\Sig$ in (\ref{eq:EQSigdef}) results
in a bounded condition number in (\ref{eq:augsys}), even as $\mu\to0^{+}$.
Notably, the well-conditioning holds for any rank parameter $r\ge\rank(X^{\star})$;
this is important, because the exact value of $\rank(X^{\star})$
is often not precisely known in practice. The following is a more
precise version of (\ref{eq:A2}) stated in a scale-invariant form. 

\begin{assume}[Tangent space injectivity]\label{asm:inject}Factor
$X^{\star}=U^{\star}U^{\star T}$ with $U^{\star}\in\R^{n\times r}$.
Then, there exists a condition number $1\le\chi_{2}<\infty$ such
that
\[
\|(\A\A^{T})^{-1/2}\A\;\svec(U^{\star}V^{T}+VU^{\star})\|\ge\chi_{2}^{-1}\cdot\|U^{\star}V^{T}+VU^{\star}\|_{F}\quad\forall V\in\R^{n\times r}.
\]
\end{assume}

\begin{theorem}[Well-conditioning]\label{thm:main}
Let both \asmref{analy} and \asmref{inject} hold with
parameters $L,\delta,\chi_{1},\chi_{2}$. Given $W\in\S_{++}^{n}$,
select $\rank(X^{\star})\le r\le n$ and $\tau=\frac{1}{2}\lambda_{r+1}(W)$,
and define $\E,\Q,\Sig$ as in (\ref{eq:EQSigdef}). For $\G$ satisfying
$\gamma_{\max}\A\A^{T}\succeq\G\succeq\gamma_{\min}\A\A^{T}$ with
$\gamma_{\max}\ge\gamma_{\min}>0$, we have
\[
\cond\left(\begin{bmatrix}\G & \A\Q\\
\Q^{T}\A^{T} & -\tau^{2}\cdot\Sig
\end{bmatrix}\right)=O(\cond(\A\A^{T})\cdot\gamma_{\max}^{4}\cdot\gamma_{\min}^{-8}\cdot L^{12}\cdot(1-\delta)^{-11}\cdot\chi_{1}^{2}\cdot\chi_{2}^{6})
\]
for all $0<\mu\le1$, with no dependence on $1/\mu$.
\end{theorem}

The most obvious way to use this result is to solve the well-conditioned
indefinite system in (\ref{eq:augsys}) using MINRES, and then use
(\ref{eq:recover}) to recover a solution to (\ref{eq:Newt-eps}).
In theory, this immediately allows us to complete an interior-point
method iteration in $O(n^{3})$ time and $O(n^{2})$ memory. \textbf{}
\begin{corollary}
\label{cor:minres}Given $W\in\S_{++}^{n},$ $\A\in\R^{m\times\frac{1}{2}n(n+1)},$
$\tilde{b}\in\R^{m},\tilde{C}\in\S^{n}$, suppose that \asmref{analy}
and \asmref{inject} hold with parameters $L,\delta,\chi_{1},\chi_{2}$.
Select $\rank(X^{\star})\le r\le n$ and $\tau=\frac{1}{2}\lambda_{r+1}(W)$,
and define $\E,\Q,\Sig$ as in (\ref{eq:EQSigdef}). Let
\[
\begin{bmatrix}u_{k}\\
v_{k}
\end{bmatrix}=\MINRES_{k}\left(\begin{bmatrix}\A\E\A^{T} & \A\Q\\
\Q^{T}\A^{T} & -\tau^{2}\cdot\Sig
\end{bmatrix},\begin{bmatrix}\tilde{b}+\tau^{2}\A\E\svec(\tilde{C})\\
\tau^{2}\cdot\Q^{T}\svec(\tilde{C})
\end{bmatrix},I,0\right)
\]
and recover $\Delta X_{k}$ and $\Delta y_{k}$ from (\ref{eq:recover}).
Then, the residual condition (\ref{eq:Newt-eps}) is satisfied in
at most
\[
k=O\left(\cond(\A\A^{T})\cdot L^{20}\cdot(1-\delta)^{-19}\cdot\chi_{1}^{2}\cdot\chi_{2}^{6}\cdot\log\left(\frac{L^{4}}{\mu\cdot\epsilon\cdot(1-\delta)^{4}}\right)\right)\text{ iterations}
\]
for all $0<\mu\le1,$ with a logarithmic dependence on $1/\mu$. Moreover,
suppose that $X\mapsto\A\svec(X)$ and $y\mapsto\A^{T}y$ can be evaluated
in at most $O(n^{3})$ time and $O(n^{2})$ storage for all $X\in\S^{n}$
and $y\in\R^{m}$. Then, this algorithm can be implemented in $O(n^{3}k)$
time and $O(n^{2})$ storage. 
\end{corollary}

In practice, while the condition number of the augmented matrix in
(\ref{eq:augsys}) does remain bounded as $\mu\to0^{+}$, it can grow
to be very large, and MINRES can still require too many iterations
to be effective. We propose an indefinite preconditioning algorithm.
Below, recall that $\PCG_{k}(A,b,P,x_{0})$ denotes the $k$-th iterate
generated by PCG when solving $Ax=b$ using preconditioner $P$ and
starting from an initial $x_{0}$.

\begin{assume}[Fast low-rank matrix-vector product]\label{asm:mvp2}For
$U,V\in\R^{n\times r}$ and $y\in\R^{m}$, the two matrix-vector products
$(U,V)\mapsto\AA(UV^{T}+VU^{T})$ and $y\mapsto[\AA^{T}(y)]U$ can
both be performed in at most $O(n^{2}r)$ time and $O(n^{2})$ storage.
\end{assume}
\begin{corollary}
\label{cor:pcg}Given $W\in\S_{++}^{n},$ $\A\in\R^{m\times\frac{1}{2}n(n+1)},$
$\tilde{b}\in\R^{m},\tilde{C}\in\S^{n}$, suppose that \asmref{analy}
and \asmref{inject} hold with parameters $L,\delta,\chi_{1},\chi_{2}$.
Select $\rank(X^{\star})\le r\le n$ and $\tau=\frac{1}{2}\lambda_{r+1}(W)$,
and define $\E,\Q,\Sig$ as in (\ref{eq:EQSigdef}). Let $v_{0}=-\Sig^{-1}\Q^{T}\tilde{c}$
where $\tilde{c}=\svec(\tilde{C})$, and let
\[
\begin{bmatrix}u_{k}\\
v_{k}
\end{bmatrix}=\PCG_{k}\left(\begin{bmatrix}\A\E\A^{T} & \A\Q\\
\Q^{T}\A^{T} & -\tau^{2}\cdot\Sig
\end{bmatrix},\begin{bmatrix}\tilde{b}+\tau^{2}\A\E\tilde{c}\\
\tau^{2}\cdot\Q^{T}\tilde{c}
\end{bmatrix},\begin{bmatrix}\beta I & \A\Q\\
\Q^{T}\A^{T} & -\tau^{2}\cdot\Sig
\end{bmatrix},\begin{bmatrix}0\\
v_{0}
\end{bmatrix}\right)
\]
where $\beta$ is chosen to satisfy $\lambda_{\min}(\A\E\A^{T})\le\beta\le\lambda_{\max}(\A\E\A^{T})$,
and recover $\Delta X_{k}$ and $\Delta y_{k}$ from (\ref{eq:recover}).
Then, the residual condition (\ref{eq:Newt-eps}) is satisfied in
at most
\[
k=\left\lceil \frac{\kappa_{E}\cdot\kappa_{\A}}{2}\log\left(\frac{\kappa_{E}\cdot\kappa_{\A}\cdot\|\tilde{b}+\A\svec(W\tilde{C}W)\|}{2\epsilon}\right)\right\rceil \text{ iterations}
\]
where $\kappa_{E}=\cond(E)=\sqrt{\cond(\E)}$ and $\kappa_{\A}=\sqrt{\cond(\A\A^{T})}$,
with no dependence on $1/\mu$. Under \asmref{mvp2}, this algorithm
can be implemented with overall cost of $O(n^{3}r^{3}k)$ time and
$O(n^{2}r^{2})$ storage. 
\end{corollary}

As we explained in the discussion around \propref{indefpcg}, the
significance of the indefinite preconditioned PCG in \corref{pcg}
is that both the indefinite matrix and the indefinite preconditioner
are well-conditioned for all $0<\mu\le1$ via \thmref{main}. Therefore,
the preconditioner will continue to work even with very small values
of $\mu$. 

In practice, the iteration count predicted in \corref{pcg} closely matches experimental observations. The rank parameter $r\ge r^{\star}$ controls
the tradeoff between the cost of the preconditioner and the reduction
in iteration count. A larger value of $r$ leads to a montonously
smaller $\cond(E)$ and faster convergence, but also a cubically higher
cost. 

\section{Proof of Well-Conditioning}

Our proof of \thmref{main} is based on the following.
\begin{lemma}
\label{lem:Schur}Suppose that $\gamma_{\max}\A\A^{T}\succeq\G\succeq\gamma_{\min}\A\A^{T}$.
Then,
\[
\cond\left(\begin{bmatrix}\G & \A\Q\\
\Q^{T}\A & -\tau^{2}\Sig
\end{bmatrix}\right)=O(\gamma_{\max}\cdot\gamma_{\min}^{-5}\cdot\cond(\A\A^{T})\cdot\cond(\C))
\]
where $\C=\tau^{2}\Sig+\Q^{T}\A^{T}\G^{-1}\A\Q$.
\end{lemma}

\begin{proof}
Block diagonal preconditioning with $\P=(\A\A^{T})^{-1/2}$ yields
\[
\M=\begin{bmatrix}\G & \A\Q\\
\Q^{T}\A^{T} & -\tau^{2}\Sig
\end{bmatrix}=\begin{bmatrix}\P^{-1} & 0\\
0 & I
\end{bmatrix}\begin{bmatrix}\P\G\P & \P\A\Q\\
\Q^{T}\A^{T}\P & -\tau^{2}\Sig
\end{bmatrix}\begin{bmatrix}\P^{-1} & 0\\
0 & I
\end{bmatrix}.
\]
Let $\tilde{\G}=\P\G\P$ and $\tilde{\A}=\P\A$. We perform a block-triangular
decomposition
\begin{align*}
\tilde{\M}=\begin{bmatrix}\tilde{\G} & \tilde{\A}\Q\\
\Q^{T}\tilde{\A}^{T} & -\tau^{2}\Sig
\end{bmatrix} & =\begin{bmatrix}I & 0\\
\Q^{T}\tilde{\A}^{T}\tilde{\G}^{-1} & I
\end{bmatrix}\begin{bmatrix}\tilde{\G} & 0\\
0 & -\C
\end{bmatrix}\begin{bmatrix}I & \tilde{\G}^{-1}\tilde{\A}\Q\\
0 & I
\end{bmatrix},
\end{align*}
where $\C=\tau^{2}\Sig+\Q^{T}\tilde{\A}^{T}\tilde{\G}\tilde{\A}\Q=\tau^{2}\Sig+\Q^{T}\A^{T}\G^{-1}\A\Q$.
Substituting $\|\tilde{\G}\|\le\gamma_{\max}$ and $\|\tilde{\G}^{-1}\|\le\gamma_{\min}^{-1}$
yields $\cond(\tilde{\M})\le(1+\gamma_{\min}^{-1})^{4}(\gamma_{\max}+\|\C\|)(\gamma_{\min}^{-1}+\|\C^{-1}\|)$.
The desired estimate follows from $\cond(\M)\le(1+\|\A\A^{T}\|)(1+\|(\A\A^{T})^{-1}\|)\cond(\tilde{\M})$.
\end{proof}

Given that $\G$ is already assumed to be well-conditioned, \lemref{Schur}
says that the augmented matrix is well-conditioned if and only if
the Schur complement $\C$ is well-conditioned. \propref{correct}
assures us that $\|\C\|$ is always bounded as $\mu\to0^{+}$, so
the difficulty of the proof is to show that $\|\C^{-1}\|$ also remains
bounded. 

In general, we do not know the true rank $r^{\star}=\rank(X^{\star})$
of the solution. If we choose $r>r^{\star}$, then both terms $\tau^{2}\Sig$
and $\Q^{T}\A^{T}\G^{-1}\A\Q$ will become singular in the limit $\mu=0^{+}$,
but their sum $\C$ will nevertheless remain non-singular. To understand
why this occurs, we need to partition the columns of $Q$ into the
dominant $r^{\star}=\rank(X^{\star})$ eigenvalues and the $r-r^{\star}$
excess eigenvalues, as in
\begin{equation}
W=\quad\underbrace{Q_{1}\Lambda_{1}Q_{1}^{T}}_{\text{top }r^{\star}\text{ eigenvalues}}\quad+\quad\underbrace{Q_{2}\Lambda_{2}Q_{2}^{T}}_{\text{next }r-r^{\star}\text{ eigenvalues}}\quad+\quad\underbrace{Q_{\perp}\Lambda_{\perp}Q_{\perp}^{T}.}_{\text{bottom }n-r\text{ eigenvalues}}\label{eq:Wsplit2}
\end{equation}
We emphasize that the partitioning $Q=[Q_{1},Q_{2}]$ is done purely
for the sake of analysis. Our key insight is that the matrix $\Q$
inherents a similar partitioning. 
\begin{lemma}[Partioning of $\Q$ and $\Sig$]
\label{lem:part}Given $W\in\S_{++}^{n}$ in (\ref{eq:Wsplit2}),
choose $0<\tau<\lambda_{r}(W)$, let $Q=[Q_{1},Q_{2}]$ and $\Lambda=\diag(\Lambda_{1},\Lambda_{2})$.
Define 
\[
\Q=[Q\skron Q,\;\sqrt{2}\Psi_{n}^{T}(Q\otimes Q_{\perp})],\quad\Sig=\diag(\Lambda\skron\Lambda-\tau^{2}I,\;(\Lambda-\tau I)\otimes\Lambda_{\perp})^{-1}.
\]
Then, there exists permutation matrix $\Pi$ so that $\Q\Pi=[\Q_{1},\Q_{2}]$
where 
\[
\Q_{1}=[Q_{1}\skron Q_{1},\;\sqrt{2}\Psi_{n}^{T}(Q_{1}\otimes[Q_{2},Q_{\perp}])],\quad\Q_{2}=[Q_{2}\skron Q_{2},\;\sqrt{2}\Psi_{n}^{T}(Q_{2}\otimes Q_{\perp})],
\]
and $\Pi^{T}\Sig\Pi=\diag(\Sig_{1},\Sig_{2})$ where 
\begin{align*}
\Sig_{1} & =\diag(\Lambda_{1}\skron\Lambda_{1}-\tau^{2}I,\quad\Lambda_{1}\otimes\diag(\Lambda_{2},\Lambda_{\perp})-\tau I\otimes\diag(I,\Lambda_{\perp}))^{-1},\\
\Sig_{2} & =\diag(\Lambda_{2}\skron\Lambda_{2}-\tau^{2}I,\quad(\Lambda_{2}-\tau I)\otimes\Lambda_{\perp})^{-1}.
\end{align*}
\end{lemma}

\begin{proof}
Let $[Q_{1},Q_{2},Q_{\perp}]=I_{n}$ without loss of generality. For
any $B\in\S^{r}$ and $N\in\R^{n\times(n-r)}$, we verify that
\begin{align*}
 & \Q\begin{bmatrix}\svec(B)\\
\sqrt{2}\vector(N)
\end{bmatrix}=\svec\left(\begin{bmatrix}B & N^{T}\\
N & 0
\end{bmatrix}\right)=\svec\left(\begin{bmatrix}B_{11} & B_{21}^{T} & N_{1}^{T}\\
B_{21} & B_{22} & N_{2}^{T}\\
N_{1} & N_{2} & 0
\end{bmatrix}\right)\\
= & \Q_{1}\begin{bmatrix}\svec(B_{11})\\
\sqrt{2}\vector\left(\begin{bmatrix}B_{21}\\
N_{1}
\end{bmatrix}\right)
\end{bmatrix}+\Q_{2}\begin{bmatrix}\svec(B_{22})\\
\vector(N_{2})
\end{bmatrix}=[\Q_{1},\Q_{2}]\Pi^{T}\begin{bmatrix}\svec(B)\\
\sqrt{2}\vector(N)
\end{bmatrix},
\end{align*}
where we have partitioned $B$ and $N$ appropriately. Next, we verify
that
\begin{align*}
 & \Q\Sig^{-1}\begin{bmatrix}\svec(B)\\
\sqrt{2}\vector(N)
\end{bmatrix}=\Q\begin{bmatrix}\svec(\Lambda B\Lambda-\tau^{2}I)\\
\sqrt{2}\vector[\Lambda_{\perp}N(\Lambda-\tau I)]
\end{bmatrix}\\
= & \Q_{1}\begin{bmatrix}\svec(\Lambda_{1}B_{11}\Lambda_{1}-\tau B_{11})\\
\sqrt{2}\vector\left(\begin{bmatrix}\Lambda_{2}B_{21}\Lambda_{1}\\
\Lambda_{\perp}N_{1}\Lambda_{1}
\end{bmatrix}-\tau\begin{bmatrix}B_{21}\\
\Lambda_{\perp}N_{1}
\end{bmatrix}\right)
\end{bmatrix}+\Q_{2}\begin{bmatrix}\svec(\Lambda_{2}B_{22}\Lambda_{2}-\tau B_{22})\\
\vector[\Lambda_{\perp}N(\Lambda_{2}-\tau I)]
\end{bmatrix}\\
= & \Q_{1}\Sig_{1}^{-1}\begin{bmatrix}\svec(B_{11})\\
\sqrt{2}\vector\left(\begin{bmatrix}B_{21}\\
N_{1}
\end{bmatrix}\right)
\end{bmatrix}+\Q_{2}\Sig_{2}^{-1}\begin{bmatrix}\svec(B_{22})\\
\vector(N_{2})
\end{bmatrix}.
\end{align*}
\end{proof}
Applying \lemref{part} allows us to further partition the Schur complement
$\C$ into blocks:
\[
\Pi^{T}\C\Pi=\begin{bmatrix}\tau^{2}\Sig_{1}\\
 & \tau^{2}\Sig_{2}
\end{bmatrix}+\begin{bmatrix}\Q_{1}^{T}\\
\Q_{2}^{T}
\end{bmatrix}\A^{T}\G^{-1}\A\begin{bmatrix}\Q_{1} & \Q_{2}\end{bmatrix}.
\]
In the limit $\mu=0^{+}$, the following two lemmas assert that the
two diagonal blocks $\tau^{2}\Sig_{2}$ and $\Q_{1}^{T}\A^{T}\G^{-1}\A\Q_{1}$
will remain nonsingular. This is our key insight for why $\C$ will
also remain nonsingular. 
\begin{lemma}[Eigenvalue bounds on $\Sig_{2}$]
\label{lem:eig2}Given $W\in\S_{++}^{n},$ choose $\tau=\frac{1}{2}\lambda_{r+1}(W)$,
and define $\Sig_{2}$ as in \lemref{part}. Under \asmref{analy},
$\tau^{2}\lambda_{\min}(\Sig_{2})\ge1/(4C_{1}^{4})$ where $C_{1}$
is defined in \lemref{eigsplit}.
\end{lemma}

\begin{proof}
Write $w_{i}\equiv\lambda_{i}(W)$. Indeed, $\tau^{-2}\lambda_{\min}^{-1}(\Sig_{2})\le4w_{r^{\star}+1}^{2}/w_{n}^{2}$ since 
\[\lambda_{\min}^{-1}(\Sig)\le\max\{\lambda_{\max}(\Lambda_{2}\otimes\Lambda_{\perp}),\lambda_{\max}(\Lambda_{2}\skron\Lambda_{2})\}\le w_{r^{\star}+1}^{2}.\]
Substituting \lemref{eigsplit} yields the desired bound.
\end{proof}
\begin{lemma}[Tangent space injectivity]
\label{lem:inject}Given $U\in\R^{n\times r}$, let $Q=\orth(U)$
and let $Q_{\perp}$ to be its orthogonal complement, and define $\Q=[\Q_{B},\Q_{N}]$
where $\Q_{B}=Q\skron Q$ and $\Q_{N}=\sqrt{2}\Psi_{n}^{T}(Q\otimes Q_{\perp})$.
Then, we have $\lambda_{\min}^{1/2}(\Q^{T}\A^{T}\A\Q)=\eta_{\A}(U)$
where
\[
\eta_{\A}(U)\eqdef\min_{V\in\R^{n\times r}}\{\|\A\svec(UV^{T}+VU^{T})\|:\|UV^{T}+VU^{T}\|_{F}=1\}.
\]
\end{lemma}

\begin{proof}
Let $\rank(U)=k$. The characterization of $\Q=[\Q_{B},\Q_{N}]$ where
$\Q_{B}=Q\skron Q$ and $\Q_{N}=\sqrt{2}\Psi_{n}^{T}(Q\otimes Q_{\perp})$
in \propref{correct} yields 
\[
\lambda_{\min}^{1/2}(\Q^{T}\A^{T}\A\Q)=\min_{\|h\|=1}\|\A\Q h\|\overset{\text{(a)}}{=}\min_{\|\Q h\|=1}\|\A\Q h\|\overset{\text{(b)}}{=}\eta_{\A}(U).
\]
Step (a) follows from $\Q^{T}\Q=I_{d}$ with $d=nk-\frac{1}{2}k(k-1)$.
Step (b) is by substituting $\colspan(\Q)=\{\svec(QV^{T}):V\in\R^{n\times k}\}=\{\svec(U\overline{V}^{T}):\overline{V}\in\R^{n\times r}\}$.
\end{proof}
In order to be able to accommodate small but nonzero values of $\mu>0$,
we will need to be able to perturb \lemref{inject}.

\begin{lemma}[Injectivity perturbation]
\label{lem:pert}\sloppy Let $Q,\hat{Q}\in\R^{n\times r}$ have orthonormal
columns. Then, $|\eta_{\A}(Q)-\eta_{\A}(\hat{Q})|\le10\|\A\|\|(I-\hat{Q}\hat{Q}^{T})Q\|.$
\end{lemma}

We will defer the proof of \lemref{pert} in order to prove our main
result. Below, we factor $X^{\star}=Q^{\star}\Lambda^{\star}Q^{\star T}$.
We will need the following claim. 
\begin{claim}
\label{claim:Lip}Under \asmref{analy}, $\|(I-Q^{\star}Q^{\star T})Q_{1}\|^{2}\le\mu\cdot\frac{C_{2}L}{\sqrt{1-\delta}}$. 
\end{claim}

\begin{proof}
\lemref{eigsplit} says $W\succeq\lambda_{r^{\star}}(W)Q_{1}Q_{1}^{T}\succeq(C_{2}\sqrt{\mu})^{-1}Q_{1}Q_{1}^{T}$
and therefore 
\begin{align*}
\frac{1}{C_{2}\sqrt{\mu}}\|Q_{1}^{T}Q_{\perp}^{\star}\|^{2} & \le\inner W{Q_{\perp}^{\star}Q_{\perp}^{\star T}}=\inner{(X^{1/2}SX^{1/2})^{-1/2}}{X^{1/2}Q_{\perp}^{\star}Q_{\perp}^{\star T}X^{1/2}}\\
 & \le\|(X^{1/2}SX^{1/2})^{-1/2}\|\cdot\inner X{Q_{\perp}^{\star}Q_{\perp}^{\star T}}\\
 & =\lambda_{\min}^{-1/2}(X^{1/2}SX^{1/2})\cdot\inner{X-X^{\star}}{Q_{\perp}^{\star}Q_{\perp}^{\star T}}\le\mu\cdot\frac{L}{\sqrt{(1-\delta)\mu}}.
\end{align*}
\end{proof}
\begin{proof}[Proof of \thmref{main}]
Write $\bK_{ij}\eqdef\Q_{i}^{T}\A^{T}\G^{-1}\A\Q_{j}$ for $i,j\in\{1,2\}$.
Under \asmref{analy}, substituting \claimref{Lip} into \lemref{pert}
with $\B=(\A\A^{T})^{-1/2}\A$ and taking $\eta_{\B}(U^{\star})=1/\chi_{2}$
from \asmref{inject} yields
\begin{align*}
\lambda_{\min}^{1/2}(\Q_{1}^{T}\B^{T}\B\Q_{1})=\eta_{\B}(Q_{1}) & \ge\eta_{\B}(Q^{\star})-5\|\B\|\|(I-Q^{\star}Q^{\star T})Q_{1}\|\\
 & \ge\frac{1}{\chi_{2}}-\mu\cdot\frac{5L}{C_{2}\sqrt{1-\delta}}.
\end{align*}
In particular, if $0<\mu\le\mu_{0}\eqdef\frac{C_{2}\sqrt{1-\delta}}{6L}$,
then $\lambda_{\min}^{1/2}(\Q_{1}^{T}\B^{T}\B\Q_{1})\ge\frac{1}{6\chi_{2}}$
and
\begin{align}
\bK_{11}=\Q_{1}^{T}\A^{T}\G^{-1}\A\Q_{1}\succeq & \frac{1}{\gamma_{\max}}\Q_{1}^{T}\A^{T}(\A\A^{T})^{-1}\A\Q_{1}\nonumber \\
= & \frac{1}{\gamma_{\max}}\Q_{1}^{T}\B^{T}\B\Q_{1}\succeq\frac{1}{36\chi_{2}^{2}\gamma_{\max}}.\label{eq:K11lb}
\end{align}
This implies $\C\succ0$ via the steps below: 
\begin{align*}
\Pi^{T}\C\Pi & =\begin{bmatrix}\tau^{2}\Sig_{1}\\
 & \tau^{2}\Sig_{2}
\end{bmatrix}+\begin{bmatrix}\bK_{11} & \bK_{21}^{T}\\
\bK_{21} & \bK_{22}
\end{bmatrix}\\
 & \succeq\begin{bmatrix}0 & 0\\
0 & \tau^{2}\Sig_{2}
\end{bmatrix}+\begin{bmatrix}I\\
\bK_{21}\bK_{11}^{-1} & I
\end{bmatrix}\begin{bmatrix}\bK_{11}\\
 & 0
\end{bmatrix}\begin{bmatrix}I & \bK_{21}\bK_{11}^{-1}\\
 & I
\end{bmatrix}\\
 & =\begin{bmatrix}I\\
\bK_{21}\bK_{11}^{-1} & I
\end{bmatrix}\begin{bmatrix}\bK_{11}\\
 & \tau^{2}\Sig_{2}
\end{bmatrix}\begin{bmatrix}I & \bK_{21}\bK_{11}^{-1}\\
 & I
\end{bmatrix}.
\end{align*}
Indeed, substituting $\|\bK_{11}^{-1}\|\le36\chi_{2}^{2}\gamma_{\max}$
from (\ref{eq:K11lb}) and $\|(\tau^{2}\Sig_{2})^{-1}\|\le4C_{1}^{4}$
from \lemref{eig2} and $C_{1}=O(L\cdot(1-\delta)^{-1})$ yields
\begin{align*}
\|\C^{-1}\| & \le(1+\|\bK_{21}\bK_{11}^{-1}\|)^{2}(\|\bK_{11}^{-1}\|+\|(\tau^{2}\Sig_{2})^{-1}\|)\\
 & \le\left(1+36\chi_{2}^{2}\gamma_{\max}\cdot\gamma_{\min}^{-1}\right)^{2}\left(36\chi_{2}^{2}\cdot\gamma_{\max}+4C_{1}^{4}\right)\\
 & =O(\gamma_{\max}^{3}\cdot\gamma_{\min}^{-2}\cdot\chi_{2}^{6}\cdot L^{4}\cdot(1-\delta)^{-4}).
\end{align*}
The second line uses the hypothesis $\G\succeq\gamma_{\min}\A\A^{T}\succ0$
to bound
\[
\|\mathbf{K}_{21}\|\le\|\A^{T}\G^{-1}\A\|\le\gamma_{\min}^{-1}\|\A^{T}(\A\A^{T})^{-1}\A\|=\gamma_{\min}^{-1}.
\]
Otherwise, if $\mu_{0}<\mu\le1$, then substituting $\|(\tau^{2}\Sig)^{-1}\|\le4C_{1}^{4}/\mu^{2}$
from \propref{correct} and $C_{2}=O(\chi_{1}\cdot L^{2}\cdot(1-\delta)^{-1})$
yields
\[
\|\C^{-1}\|\le\|(\tau^{2}\Sig)^{-1}\|\le4C_{1}^{4}/\mu_{0}^{2}=4C_{1}^{4}\cdot\frac{36L^{2}C_{2}^{2}}{1-\delta}=O(L^{8}\cdot\chi_{1}^{2}\cdot(1-\delta)^{-7}).
\]
Finally, for all $0<\mu\le1$, 
\[
\|\C\|\le\|\tau^{2}\Sig\|+\|\Q^{T}\A^{T}\G^{-1}\A\Q\|\le C_{1}^{4}+\gamma_{\min}^{-1}=O(\gamma_{\min}^{-1}\cdot L^{4}\cdot(1-\delta)^{-4}).
\]
Substituting $\|\C\|$ and $\|\C^{-1}\|$ into \lemref{Schur} with
the following simplification
\[
\xi(1+\gamma_{\min}^{-1})^{4}(\gamma_{\max}+\|\C\|)(\gamma_{\min}^{-1}+\|\C^{-1}\|)=O(\gamma_{\max}\cdot\gamma_{\min}^{-5}\cdot\cond(\A\A^{T})\cdot\cond(\C))
\]
yields the desired bound. 
\end{proof}
We now turn to the proof of \lemref{pert}. We first need to prove
a technical lemma on the distance between orthogonal matrices. 
\global\long\def\polar{\operatorname{polar}}%
\global\long\def\Orth{\operatorname{Orth}}%
Below, $\Orth(n)$ denotes the set of $n\times n$ orthonormal matrices.
\begin{lemma}
\label{lem:Qpert}Let $Q,\hat{Q}\in\R^{n\times r}$ and $Q_{\perp},\hat{Q}_{\perp}\in\R^{n\times(n-r)}$
satisfy $[Q,Q_{\perp}]\in\Orth(n)$ and $[\hat{Q},\hat{Q}_{\perp}]\in\Orth(n)$.
Then, 
\begin{align*}
\max\left\{ \min_{R\in\Orth(r)}\|Q-\hat{Q}R\|,\min_{R_{\perp}\in\Orth(n-r)}\|Q_{\perp}-\hat{Q}_{\perp}R_{\perp}\|\right\}  & \le\sqrt{2}\|Q^{T}\hat{Q}_{\perp}\|.
\end{align*}
\end{lemma}

\begin{proof}
Let $\sigma_{\max}(\hat{Q}_{\perp}^{T}Q)=\sin\theta.$ Define $\polar(M)=\arg\max_{R\in\Orth(n)}\inner MR=UV^{T}$
where $M=U\Sig V^{T}$ is the usual SVD. The choice of $R=\polar(\hat{Q}^{T}Q)$
yields 
\begin{align*}
\|Q-\hat{Q}R\|^{2} & =\|\hat{Q}-QR^{T}\|^{2}\le\|I-\hat{Q}^{T}QR^{T}\|^{2}+\|\hat{Q}_{\perp}^{T}QR^{T}\|^{2}\\
 & \overset{\text{(a)}}{=}[1-\sigma_{\min}(\hat{Q}^{T}Q)]^{2}+\sigma_{\max}^{2}(\hat{Q}_{\perp}^{T}Q)\\
 & \overset{\text{(b)}}{=}(1-\cos\theta)^{2}+\sin^{2}\theta\le2\sin^{2}(\theta)\text{ for }\theta\in[0,\pi/2].
\end{align*}
\sloppy Step (a) is because our choice of $R$ renders $\hat{Q}^{T}QR^{T}=RQ^{T}\hat{Q}\succeq0$.
Step (b) is because of $\sigma_{\min}(\hat{Q}^{T}Q)=\cos\theta$,
which follows from $\lambda_{\min}(Q^{T}\hat{Q}\hat{Q}^{T}Q)=\lambda_{\min}(Q^{T}[I-\hat{Q}_{\perp}\hat{Q}_{\perp}^{T}]Q)=1-\lambda_{\max}(Q^{T}\hat{Q}_{\perp}\hat{Q}_{\perp}^{T}Q)$.
The proof for $\|Q_{\perp}-\hat{Q}_{\perp}R_{\perp}\|^{2}\le2\sin^{2}(\theta)$
follows identical steps with the choice of $R_{\perp}=\polar(\hat{Q}_{\perp}^{T}Q_{\perp})$.
\end{proof}
We will now use \lemref{Qpert} to prove \lemref{pert}.
\begin{proof}[Proof of \lemref{pert}]
Given that $\eta(Q)=\eta(QR)$ for $R\in\Orth(r)$, we assume without
loss of generality that $\|Q-\hat{Q}\|\le\sqrt{2}\|Q^{T}\hat{Q}_{\perp}\|$
and $\|Q_{\perp}-\hat{Q}_{\perp}\|\le\sqrt{2}\|Q^{T}\hat{Q}_{\perp}\|$
as in \lemref{Qpert}. Applying Weyl's inequality for singular values
yields
\[
|\eta(Q)-\eta(\hat{Q})|=|\sigma_{d}(\A\Q)-\sigma_{d}(\A\hat{\Q})|\le\|\A\Q-\A\hat{\Q}\|\le\|\A\|\|\Q-\hat{\Q}\|.
\]
Next, observe that 
\begin{align*}
(\Q-\hat{\Q})\begin{bmatrix}\svec(B)\\
\frac{1}{\sqrt{2}}\vector(N)
\end{bmatrix} & =\svec\left(\begin{bmatrix}Q & Q_{\perp}\end{bmatrix}\begin{bmatrix}B\\
N
\end{bmatrix}Q^{T}-\begin{bmatrix}\hat{Q} & \hat{Q}_{\perp}\end{bmatrix}\begin{bmatrix}B\\
N
\end{bmatrix}\hat{Q}^{T}\right)\\
 & =\svec\left(\begin{bmatrix}Q & Q_{\perp}\end{bmatrix}\begin{bmatrix}B\\
N
\end{bmatrix}(Q-\hat{Q})^{T}\right)\\
 & \qquad+\svec\left(\left(\begin{bmatrix}Q & Q_{\perp}\end{bmatrix}-\begin{bmatrix}\hat{Q} & \hat{Q}_{\perp}\end{bmatrix}\right)\begin{bmatrix}B\\
N
\end{bmatrix}\hat{Q}^{T}\right),
\end{align*}
and therefore
\begin{align*}
\frac{1}{\sqrt{2}}\|\Q-\hat{\Q}\| & \le\|Q-\hat{Q}\|+\|[Q,Q_{\perp}]-[\hat{Q},\hat{Q}_{\perp}]\|\\
 & \le\|Q-\hat{Q}\|+\sqrt{\|Q-\hat{Q}\|^{2}+\|Q_{\perp}-\hat{Q}_{\perp}\|^{2}}\le(1+\sqrt{2})\cdot\|Q^{T}\hat{Q}_{\perp}\|.
\end{align*}
Finally, we note that $(1+\sqrt{2})\cdot2\approx4.828<5$.
\end{proof}

\section{Solution via MINRES}

Let us estimate the cost of solving (\ref{eq:augrec}) using MINRES,
as in \corref{minres}. In the initial setup, we spend $O(n^{3})$
time and $O(n^{2})$ memory to explicitly compute the diagonal matrix
$\Sig$, and to implicitly set up $\E,\Q$ via their constituent matrices
$E\in\S^{n}$, $Q\in\R^{n\times r},$ and $Q_{\perp}\in\R^{n\times(n-r)}$.
Afterwards, each matrix-vector product with $\E,\Q,\Q^{T}$ can be
evaluated in $O(n^{3})$ time and $O(n^{2})$ memory, because $\E\svec(X)=\svec(EXE)$
and
\[
\Q\begin{bmatrix}\svec(B)\\
\vector(N)
\end{bmatrix}=\svec\left(\begin{bmatrix}Q & Q_{\perp}\end{bmatrix}\begin{bmatrix}B\\
\sqrt{2}N
\end{bmatrix}Q^{T}\right),\quad\Q^{T}\svec(X)=\begin{bmatrix}\svec(Q^{T}XQ)\\
\sqrt{2}\vector(Q_{\perp}^{T}XQ)
\end{bmatrix}.
\]
Therefore, the per-iteration cost of MINRES, which is dominated by
a single matrix-vector product with the augmented matrix in (\ref{eq:augsys}),
is $O(n^{3})$ time and $O(n^{2})$ memory, plus a single matrix-vector
product with $\A$ and $\A^{T}$. This is also the cost of evaluating
the recovery equation in (\ref{eq:recover}).

We now turn to the number of iterations. \propref{minres} says that
MINRES converges to $\epsilon$ residual in $O(\kappa\log(1/\epsilon))$
iterations, where $\kappa$ is the condition estimate from \thmref{main},
but we need the following lemma to ensure that small residual in (\ref{eq:augsys})
would actually recover through (\ref{eq:recover}) a small-residual
solution $\Delta X,\Delta y$ to our original problem (\ref{eq:Newt-eps}).
\begin{lemma}[Propagation of small residuals]
\label{lem:smallres}Given $\D=\Q\Sig^{-1}\Q^{T}+\tau^{2}\cdot\E$
where $\Q^{T}\Q=I$ and $\E\succ0$ and $\Sig\succ0$. Let
\begin{equation}
\begin{bmatrix}\A\E\A^{T} & \A\Q\\
\Q^{T}\A^{T} & -\tau^{2}\Sig
\end{bmatrix}\begin{bmatrix}u\\
v
\end{bmatrix}-\begin{bmatrix}b+\tau^{2}\cdot\A\E c\\
\tau^{2}\cdot\Q^{T}c
\end{bmatrix}=\begin{bmatrix}p\\
d
\end{bmatrix}.\label{eq:resid1}
\end{equation}
Then, $y=\tau^{-2}\cdot u$ and $x=\E(\A^{T}u-\tau^{2}c)+\Q v$ satisfy
the following
\begin{equation}
\begin{bmatrix}-\D^{-1} & \A^{T}\\
\A & 0
\end{bmatrix}\begin{bmatrix}x\\
y
\end{bmatrix}-\begin{bmatrix}c\\
b
\end{bmatrix}=\begin{bmatrix}\tau^{-2}\E^{-1}\Q\C^{-1}d\\
p
\end{bmatrix}\label{eq:resid2}
\end{equation}
where $\C=\tau^{2}\Sig+\Q^{T}\E^{-1}\Q$.
\end{lemma}

\begin{proof}
Define $(x^{\star},u^{\star},v^{\star})$ as the exact solution to
the following system of equations
\[
\begin{bmatrix}-\E^{-1} & \A^{T} & \E^{-1}\Q\\
\A & 0 & 0\\
\Q^{T}\E^{-1} & 0 & -\C
\end{bmatrix}\begin{bmatrix}x^{\star}\\
u^{\star}\\
v^{\star}
\end{bmatrix}=\begin{bmatrix}\tau^{2}c\\
b\\
0
\end{bmatrix}.
\]
Our key observation is that the matrix has two possible block factorizations
\begin{align}
 & \begin{bmatrix}I\\
-\A\E & I & 0\\
-\Q^{T} & 0 & I
\end{bmatrix}\begin{bmatrix}-\E^{-1}\\
 & \A\E\A^{T} & \A\Q\\
 & \Q^{T}\A^{T} & -\tau^{2}\Sig
\end{bmatrix}\begin{bmatrix}I & -\E\A^{T} & -\Q\\
 & I & 0\\
 & 0 & I
\end{bmatrix}\label{eq:tri1}\\
= & \begin{bmatrix}I & 0 & -\E^{-1}\Q\C^{-1}\\
0 & I & 0\\
 &  & I
\end{bmatrix}\begin{bmatrix}-\tau^{2}\D^{-1} & \A^{T}\\
\A & 0\\
 &  & -\C
\end{bmatrix}\begin{bmatrix}I & 0\\
0 & I\\
-\C^{-1}\Q^{T}\E^{-1} &  & I
\end{bmatrix},\label{eq:tri2}
\end{align}
where $\tau^{2}\D^{-1}=\E^{-1}-\E^{-1}\Q\C^{-1}\Q^{T}\E^{-1}$ is
via the Sherman--Morrison--Woodbury identity. We verify from (\ref{eq:tri1})
and (\ref{eq:tri2}) respectively that 
\[
\begin{bmatrix}\A\E\A^{T} & \A\Q\\
\Q^{T}\A^{T} & -\tau^{2}\Sig
\end{bmatrix}\begin{bmatrix}u^{\star}\\
v^{\star}
\end{bmatrix}=\begin{bmatrix}b+\tau^{2}\cdot\A\E c\\
\tau^{2}\cdot\Q^{T}c
\end{bmatrix},\quad\begin{bmatrix}-\tau^{2}\D^{-1} & \A^{T}\\
\A & 0
\end{bmatrix}\begin{bmatrix}x^{\star}\\
u^{\star}
\end{bmatrix}=\begin{bmatrix}\tau^{2}c\\
b
\end{bmatrix},
\]
and $x^{\star}=\E\A^{T}u^{\star}+\Q v^{\star}$. Now, for a given
$(x,u,v)$, the associated errors are
\[
\begin{bmatrix}\A\E\A^{T} & \A\Q\\
\Q^{T}\A^{T} & -\tau^{2}\Sig
\end{bmatrix}\begin{bmatrix}u-u^{\star}\\
v-v^{\star}
\end{bmatrix}=\begin{bmatrix}p\\
d
\end{bmatrix},\qquad x-x^{\star}=\begin{bmatrix}\E\A^{T} & \Q\end{bmatrix}\begin{bmatrix}u-u^{\star}\\
v-v^{\star}
\end{bmatrix},
\]
and these can be rearranged as
\[
\begin{bmatrix}-\E^{-1}\\
 & \A\E\A^{T} & \A\Q\\
 & \Q^{T}\A^{T} & -\tau^{2}\Sig
\end{bmatrix}\begin{bmatrix}I & -\E\A^{T} & -\Q\\
 & I & 0\\
 & 0 & I
\end{bmatrix}\begin{bmatrix}x-x^{\star}\\
u-u^{\star}\\
v-v^{\star}
\end{bmatrix}=\begin{bmatrix}0\\
p\\
d
\end{bmatrix}.
\]
It follows from our first block-triangular factorization (\ref{eq:tri1})
that
\[
\begin{bmatrix}-\E^{-1} & \A^{T} & \E^{-1}\Q\\
\A & 0 & 0\\
\Q^{T}\E^{-1} & 0 & -\C
\end{bmatrix}\begin{bmatrix}x-x^{\star}\\
u-u^{\star}\\
v-v^{\star}
\end{bmatrix}=\begin{bmatrix}I\\
-\A\E & I & 0\\
-\Q^{T} & 0 & I
\end{bmatrix}\begin{bmatrix}0\\
p\\
d
\end{bmatrix}=\begin{bmatrix}0\\
p\\
d
\end{bmatrix}.
\]
It follows from our second block-triangular factorization (\ref{eq:tri2})
that
\[
\begin{bmatrix}-\tau^{2}\D^{-1} & \A^{T}\\
\A & 0\\
 &  & -\C
\end{bmatrix}\begin{bmatrix}I & 0\\
0 & I\\
-\C^{-1}\Q^{T}\E^{-1} &  & I
\end{bmatrix}\begin{bmatrix}x-x^{\star}\\
u-u^{\star}\\
v-v^{\star}
\end{bmatrix}=\begin{bmatrix}\E^{-1}\Q\C^{-1}d\\
p\\
d
\end{bmatrix}.
\]
We obtain (\ref{eq:resid2}) by isolating the first two block-rows,
rescaling the first block-row by $\tau^{-2}$, and then substituting
$(x^{\star},u^{\star})$ as an exact solution.
\end{proof}
It follows that the residuals between (\ref{eq:augrec}) and (\ref{eq:Newt-eps})
are related as
\[
\left\Vert \begin{bmatrix}-\D^{-1} & \A^{T}\\
\A & 0
\end{bmatrix}\begin{bmatrix}x\\
y
\end{bmatrix}-\begin{bmatrix}b\\
c
\end{bmatrix}\right\Vert \le\frac{16C_{1}^{4}C_{2}^{2}}{\mu}\cdot\left\Vert \begin{bmatrix}\A\E\A^{T} & \A\Q\\
\Q^{T}\A^{T} & -\tau^{2}\Sig
\end{bmatrix}\begin{bmatrix}u\\
v
\end{bmatrix}-\begin{bmatrix}b+\tau^{2}\cdot\A\E c\\
\tau^{2}\cdot\Q^{T}c
\end{bmatrix}\right\Vert 
\]
where we used $\|\tau^{-2}\E^{-1}\Q\C^{-1}\|\le\tau^{-2}\|\E^{-1}\|\|\E\|,$
and substituted $\tau=\frac{1}{2}\lambda_{r+1}(W)\ge\sqrt{\mu}/C_{1}$
where $C_{1}=\frac{1+L}{1-\delta}$ from \lemref{eigsplit}, and $\|\E\|\le4$
and $\|\E^{-1}\|\le C_{1}^{4}$ from \propref{correct}. We obtain
\corref{minres} by substituting \thmref{main} into \propref{minres},
setting $\gamma_{\max}=\|\E\|=4$ and $\gamma_{\min}^{-1}=\|\E^{-1}\|=C_{1}^{4}$
and then performing enough iterations to achieve a residual of $\frac{\mu\epsilon}{16C_{1}^{4}}$. 

\section{Solution via indefinite PCG}

Let us estimate the cost of solving (\ref{eq:augrec}) using indefinite
PCG, as in \corref{pcg}. After setting up $\E,\Q,\Q^{T}$ in the
same way as using MINRES, we spend $O(d\cdot n^{2}r)=O(n^{3}r^{2})$
time and $O(d^{2})=O(n^{2}r^{2})$ memory to explicitly form the $d\times d$
Schur complement of the preconditioner
\[
\C=\tau^{2}\Sig+\beta^{-1}\Q^{T}\A^{T}\A\Q,
\]
where $d=nr-\frac{1}{2}r(r-1)$. This can be done using $d$ matrix-vector products with $\A\Q$ and
$\Q^{T}\A^{T}$, which under \asmref{mvp2}, can each be evaluated
in at $O(n^{2}r)$ time and $O(n^{2})$ memory:
\begin{gather*}
\A\Q\begin{bmatrix}\svec(B)\\
\vector(N)
\end{bmatrix}=\AA(QU^{T}+UQ^{T}),\qquad\Q^{T}\A^{T}y=\begin{bmatrix}\svec(Q^{T})\\
\sqrt{2}\vector(Q_{\perp}^{T}V)
\end{bmatrix}
\end{gather*}
where $U=QB+\sqrt{2}Q_{\perp}N$ and $V=\AA^{T}(y)Q$. After forming the Schor complement $\C$, the indefinite preconditioner can be applied in its block
triangular form 
\[
\begin{bmatrix}\beta I & \A\Q\\
\Q^{T}\A^{T} & -\tau^{2}\Sig
\end{bmatrix}^{-1}=\begin{bmatrix}I & -\A\Q\\
0 & I
\end{bmatrix}\begin{bmatrix}\beta I & 0\\
0 & -\C^{-1}
\end{bmatrix}\begin{bmatrix}I & 0\\
-\Q^{T}\A^{T} & I
\end{bmatrix}
\]
as a linear solve with $\C$, and a matrix-vector products with each
of $\A\Q$ and $\Q^{T}\A^{T}$, for $O(d^{3})=O(n^{3}r^{3})$ time
and $O(d^{2})=O(n^{2}r^{2})$ memory. Moreover, note that under \asmref{mvp2},
each matrix-vector product with $\A$ and $\A^{T}$ costs $O(n^{3})$
time and $O(n^{2})$ memory. Therefore, the per-iteration cost of
PCG, which is dominated by a single matrix-vector product with the
augmented matrix in (\ref{eq:augsys}) and a single application of
the indefinite preconditioner, is $O(n^{3}r^{3})$ time and $O(n^{2}r^{2})$
memory.

We now turn to the number of iterations. \propref{indefpcg} says
that indefinite PCG will generate identical iterates $u_{k}$ to regular
PCG applied to the following problem
\[
\underbrace{\A(\E+\tau^{-2}\Q\Sig^{-1}\Q^{T})\A^{T}}_{\H}u_{k}=\underbrace{b+\tau^{2}\A(\E+\tau^{-2}\Q\Sig^{-1}\Q^{T})c}_{g}
\]
with preconditioner $\tilde{\H}=\beta I+\tau^{-2}\A\Q\Sig^{-1}\Q^{T}\A^{T}$
and iterates $v_{k}$ that exactly satisfy 
\[
\Q^{T}\A^{T}u_{k}-\tau^{2}\cdot\Sig v_{k}=\tau^{2}\cdot\Q^{T}c.
\]
Plugging these iterates back into (\ref{eq:augsys}) yields
\[
\begin{bmatrix}\A\E\A^{T} & \A\Q\\
\Q^{T}\A^{T} & -\tau^{2}\Sig
\end{bmatrix}\begin{bmatrix}u_{k}\\
v_{k}
\end{bmatrix}-\begin{bmatrix}b+\tau^{2}\cdot\A\E c\\
\tau^{2}\cdot\Q^{T}c
\end{bmatrix}=\begin{bmatrix}\H u_{k}-g\\
0
\end{bmatrix},
\]
and plugging into \lemref{smallres} and \propref{pcg} shows that
the $\Delta X,\Delta y$ recovered from (\ref{eq:recover}) will satisfy
\[
\left\Vert \begin{bmatrix}-\D^{-1} & \A^{T}\\
\A & 0
\end{bmatrix}\begin{bmatrix}x\\
y
\end{bmatrix}-\begin{bmatrix}b\\
c
\end{bmatrix}\right\Vert =\|\H u_{k}-g\|\le2\sqrt{\kappa}\left(\frac{\sqrt{\kappa}-1}{\sqrt{\kappa}+1}\right)^{k}\|g\|
\]
where $\kappa=\cond(\tilde{\H}^{-1/2}\H\tilde{\H}^{-1/2})$. We use
the following argument to estimate $\kappa$.
\begin{claim}
Let $\E\succ0$ satisfy $\lambda_{\min}(\E)\le\beta\le\lambda_{\max}(\E)$.
Then 
\[
\cond((\beta I+\U\U^{T})^{-1/2}(\E+\U\U^{T})(\beta I+\U\U^{T})^{-1/2})\le\cond(\E).
\]
\end{claim}

\begin{proof}
We have $\E+\U\U^{T}\preceq\lambda_{\max}(\E)I+UU^{T}\preceq(\lambda_{\max}(\E)/\beta)(\beta I+UU^{T})$
because $\lambda_{\max}(\E)/\beta\ge1$. We have $\E+\U\U^{T}\succeq\lambda_{\min}(\E)I+UU^{T}\succeq(\lambda_{\min}(\E)/\beta)(\beta I+UU^{T})$
because $\lambda_{\min}(\E)/\beta\le1$.
\end{proof}
Therefore, given that $\lambda_{\min}(\A\E\A^{T})\le\beta\le\lambda_{\max}(\A\E\A^{T})$
holds by hypothesis, it takes at most $k$ iterations to achieve $\|\H u_{k}-g\|\le\epsilon$,
where 
\[
k=\left\lceil \frac{1}{2}\sqrt{\cond(\A\E\A^{T})}\log\left(\frac{\cond(\A\E\A^{T})\cdot\|g\|}{2\epsilon}\right)\right\rceil .
\]
Finally, we bound $\cond(\A\E\A^{T})\le\cond(\E)\cdot\cond(\A\A^{T})$
and $\cond(\E)\le\cond(E)^{2}$.

\section{Experiments}
We perform extensive experiments on challenging instances of \emph{robust matrix completion} (RMC) and \emph{sensor network localization} (SNL) problems, where the number of constraints is large, i.e., $m=O(n^{2})$. We show that these two problems satisfy our theoretical assumptions and therefore our proposed method can solve them to high accuracies in $O(n^3r^3)$ time and $O(n^2r^2)$ memory. This is a significant improvement over general-purpose IPM solvers, which often struggle to solve SDP instances with $m=O(n^{2})$ constraints to high accuracies.  Direct methods are computationally expensive, requiring $O(n^{6})$ time and $O(n^{4})$ memory to solve (\ref{eq:newt}), and iterative methods are also ineffective, as they require $O(1/\mu)$ iterations to solve (\ref{eq:newt}) due to the ill-conditioned in $W$. In contrast, our method reformulates (\ref{eq:newt}) into a well-conditioned indefinite system that can be solved in $O(n^3r^3)$ time and $O(n^2r^2)$ memory, ensuring rapid convergence to a high accuracy solution.

All of our experiments are performed on an Apple laptop in MATLAB R2021a, running a silicon M1 pro chip with 10-core CPU, 16-core GPU, and 32GB of RAM. We implement our proposed method, which we denoted as \texttt{Our method}, by modifying the source code of \texttt{SeDuMi}~\cite{sturm1999sedumi} to use (\ref{eq:augrec}) to solve (\ref{eq:Newt-eps}) while keeping all other parts unchanged. Specifically, at each iteration, we decompose the scaling point $W$ as in (\ref{eq:EQSigdef}) with the rank parameter $r=\arg\max_{i\leq \hat r}\lambda_i(W)/\lambda_{i+1}(W)$ for some fixed $\hat r$, and then solve (\ref{eq:augrec}) using \texttt{PCG} as in Corollary~\ref{cor:pcg} with $\beta$ chosen as the square of median eigenvalue of $E$. We choose to modify \texttt{SeDuMi} in order to be consistent with the prior approach of Zhang and Lavaei \cite{zhang2017modified}, which we denoted as \texttt{ZL17}.

In Section~\ref{sec:verify_thm}, we empirically verify our theoretical results using RMC and SNL problems, which include Assumptions~\ref{asm:analy}, \ref{asm:inject} and \ref{asm:mvp2}; the well-conditioned indefinite system predicted in Theorem~\ref{thm:main}; the bounded \texttt{pcg} iterations predicted in Corollary~\ref{cor:pcg}; and the cubic time complexity. In Section~\ref{sec:comparison}, we compare the practical performance of \texttt{Our method} for solving RMC and SNL problems.

\paragraph{Robust matrix completion.}
Given a linear operator $\AA:\R^{n\times n}\to\R^{m}$, the task of robust matrix completion~\cite{candes2011robust,ma2023global,ding2021rank} seeks to recover a $n\times n$ low-rank matrix $X^\star$ from highly corrupted measurements $b=\AA(X^\star)+\varepsilon$, where $\varepsilon\in\R^{m}$ is a sparse outlier noise, meaning that its nonzero entries can have arbitrary large magnitudes. RMC admits the following optimization problem
\begin{equation}
\min_{X\in\R^{n\times n}}\quad\|\AA(X)-b\|_{1}+\lambda\|X\|_{*},\label{eq:rmc}
\end{equation}
where the $\ell_{1}$ norm $\|\varepsilon\|_{1}\equiv\sum_{i}|\varepsilon_{i}|$ is used to encourage sparsity in $\varepsilon$, the nuclear norm regularizer $\|X\|_{*}\equiv\sum_{i}\sigma_{i}(X)$ is used to encourage lower rank in $X$, and $\lambda > 0$ is the regularization parameter. In order to turn (\ref{eq:rmc}) into the standard form SDP, we focus on symmetric, positive semidefinite variance of (\ref{eq:rmc}), meaning that $X\succeq 0$, and the problem is written
\begin{equation}
\min_{X\in\S^n,v,w\in\R^{m}}\quad \mathbf{1}^{T}v + \mathbf{1}^{T}w + \lambda\cdot\tr(X)\quad\text{s.t.}\quad\begin{array}{c}
v-w+\mathcal{A}(X)=b,\\
X\succeq 0,\  v\ge0,\  w\ge0.
\end{array}\label{eq:rmc-sdp}
\end{equation}

In the remainder of this section, we set $\lambda=1$ and $\AA$ according to the ``matrix completion'' model, meaning that $\mathcal{A}(X)=P_{m}\svec(X)$ where $P_{m}$ is the matrix that randomly permutes and subsamples $m$ out of $n(n+1)/2$ elements in $\svec(X)$. In each case, the ground truth is chosen as $X^{\star}=GG^{T}$ in which $G\in\R^{n\times r^\star}$ is orthogonal, and its elements are sampled as $\vector(G)\sim\mathcal{N}(0,10^2I_{nr^\star})$. We fix $r^\star=2$ in all simulations. A small number $m_{o}$ of Gaussian outliers $\varepsilon$ are added to the measurements, meaning that $b=\AA(X^{\star})+\varepsilon$ where $\varepsilon_{i}=0$ except $m_{o}$ randomly sampled elements with $\varepsilon_{i}\overset{\mathrm{i.i.d.}}{\sim}\mathcal{N}(0,10^{4})$.

\paragraph{Sensor Network Localization.} 
The task of sensor network localization is to recover the locations of $n$ sensors $y_1^\star,\dots,y_n^\star\in\R^{d}$ given the location of $k$ anchors $a_1,\dots,a_k\in\R^{d}$; the Euclidean distance $\bar \delta_{ij}$ between $y_i^\star$ and $a_j$ for some $i,j$; and the Euclidean distance $\delta_{ij}$ between $y_i^\star$ and $y_j^\star$ for some $i\neq j$. Formally
\begin{align}
\begin{split}
    \min_{y_1,\ldots,y_n} 0 \quad\text{s.t.}\quad \|y_i-y_j\|=\delta_{ij}\text{ for $i,j\in N_y$},\ \|a_i-y_j\|=\bar \delta_{ij}\text{ for $i,j\in N_a$},
    \end{split}\label{eq:snl}
\end{align}
where $N_a$ and $N_y$ denote the set of indices $(i,j)$ for which $\bar \delta_{ij}$ and $\delta_{ij}$ is specified, respectively. In general, (\ref{eq:snl}) is a nonconvex problem and difficult to solve directly, but its solution can be approximated through the following SDP relaxation \cite{alfakih1999solving,biswas2004semidefinite}
\begin{align}
\begin{split}
    \min_{X\in\S^{n+d}} & \ 0 \\
    \text{s.t }\ & \begin{bmatrix}0\\e_i-e_j\end{bmatrix}^T\begin{bmatrix}I_d&Y\\Y^T&Z\end{bmatrix}\begin{bmatrix}0\\e_i-e_j\end{bmatrix}=\delta_{ij}^2\quad\text{for $i,j\in N_y$},\\
    & \begin{bmatrix}a_i\\-e_j\end{bmatrix}^T\begin{bmatrix}I_d&Y\\Y^T&Z\end{bmatrix}\begin{bmatrix}a_i\\-e_j\end{bmatrix}=\bar \delta_{ij}^2\quad\text{for $i,j\in N_a$},\\
    &X=\begin{bmatrix}I_d&Y\\Y^T&Z\end{bmatrix}\succeq 0.
\end{split}\label{eq:snl1}
\end{align}

The sensor locations can be exactly recovered through $Y^\star=[y_1^\star,\ldots,y_n^\star]$ when the SDP relaxation (\ref{eq:snl}) is tight, which occurs if and only if $\rank(X)=d$ \cite{so2007theory}. In this experiment, we also consider a setting of SNL problem for which the outliers are presented, meaning that the distance measurements $\delta_{ij}=\|y_i^\star-y_j^\star+\varepsilon_{ij}\|$ and $\bar \delta_{ij}=\|a_i-y_j^\star+\bar\varepsilon_{ij}\|$ are corrupted by outlier noise $\varepsilon_{ij}$ and $\bar\varepsilon_{ij}$. Here, $\varepsilon_{ij}$ and $\bar\varepsilon_{ij}$ are nonzero only for a very small subset of pairs $(i,j)\in N_y$ and $(i,j)\in N_a$, respectively. The SNL problem with outlier measurements can be solved via the following SDP

\begin{align}
\begin{split}
    \min_{\substack{X\in\S^{n+d},\\u_{ij},v_{ij},\bar u_{ij},\bar v_{ij}\in\R}} & \ \sum_{(i,j)\in N_y} (v_{ij}+u_{ij}) + \sum_{(i,j)\in N_a} 
    (\bar v_{ij}+\bar u_{ij}) \\
    \text{s.t }\ & \begin{bmatrix}0\\e_i-e_j\end{bmatrix}^T\begin{bmatrix}I_d&Y\\Y^T&Z\end{bmatrix}\begin{bmatrix}0\\e_i-e_j\end{bmatrix} + v_{ij} - u_{ij}=\delta_{ij}^2\quad\text{for $i,j\in N_y$},\\
    & \begin{bmatrix}a_i\\-e_j\end{bmatrix}^T\begin{bmatrix}I_d&Y\\Y^T&Z\end{bmatrix}\begin{bmatrix}a_i\\-e_j\end{bmatrix} + \bar v_{ij} - \bar u_{ij} =\bar \delta_{ij}^2\quad\text{for $i,j\in N_a$},\\
    &X=\begin{bmatrix}I_d&Y\\Y^T&Z\end{bmatrix}\succeq 0,\ u_{ij}\geq 0,\ v_{ij}\geq 0,\ \bar u_{ij}\geq 0,\ \bar v_{ij}\geq 0.
\end{split}\label{eq:snl-sdp}
\end{align}

In the remainder of this section, we generate ground truth sensor locations $y_1^\star,\ldots,y_n^\star\in\R^{d}$ and anchor locations $a_1,\ldots,a_k\in\R^d$ from the standard uniform distribution. We fix $d=2$ and $k=3$ for all simulations. We randomly select $m_y$ out of $n(n-1)/2$ indices $(i,j)$ to be in $N_y$, and $m_a$ out of $nk$ indices $(i,j)$ to be in $N_a$. Gaussian outliers $\varepsilon$ and $\bar\varepsilon$ are added to the distance measurement $\delta_{ij}=\|y_i^\star-y_j^\star+\varepsilon_{ij}\|$ for all $(i,j)\in N_y$ and $\bar \delta_{ij}=\|a_i-y_j^\star+\bar\varepsilon_{ij}\|$ for all $(i,j)\in N_a$, respectively. Here,  $\varepsilon_{ij}=0$ except $m_{yo}$ randomly sampled elements with $\varepsilon_{ij}\overset{\mathrm{i.i.d.}}{\sim}\mathcal{N}(0,1)$, and $\bar\varepsilon_{ij}=0$ except $m_{ao}$ randomly sampled elements with $\bar\varepsilon_{ij}\overset{\mathrm{i.i.d.}}{\sim}\mathcal{N}(0,1)$. For simplicity, we use $m=m_y+m_a$ to denote the number of constraints and $m_o=m_{yo}+m_{ao}$ to denote the number of outliers in (\ref{eq:snl-sdp}).

\subsection{Validation of our theoretical results}\label{sec:verify_thm}
\paragraph{Assumptions~\ref{asm:analy} and \ref{asm:inject}.}
We start by verifying Assumptions~\ref{asm:analy} and \ref{asm:inject} hold for both RMC (\ref{eq:rmc-sdp}) and SNL (\ref{eq:snl-sdp}) problems. In Figure~\ref{fig:asm}, we empirically verify both assumptions using a small-scaled RMC problem, which is generated with $(n,m,m_o)=(50,1200,20)$, and a small-scaled SNL problem, which is generated with $(n,m,m_o)=(50,1300,25)$. As shown in Figure~\ref{fig:asm}, both assumptions hold for the RMC and SNL problems with $(\delta,L,\chi_2)=(0.9,10,2.7)$ and $(\delta,L,\chi_2)=(0.9,5,30)$, respectively. For these results, we run \texttt{Our method} and extract the primal variable $X$, dual variable $S$, and scaling point $W$ at each iteration. For each $X$, $S$ and $W$, the corresponding $\mu$ is chosen such that $\|\frac{1}{\mu}X^{1/2}SX^{1/2}-I\|=0.1$, and the corresponding matrix $\Q$ is constructed as in (\ref{eq:EQSigdef}) with the rank parameter $r=\arg\max_{i\leq 5}\lambda_i(W)/\lambda_{i+1}(W)$.

\begin{figure}[b!]
    \center
    \vspace{-0.5em}
    \includegraphics[width=\textwidth]{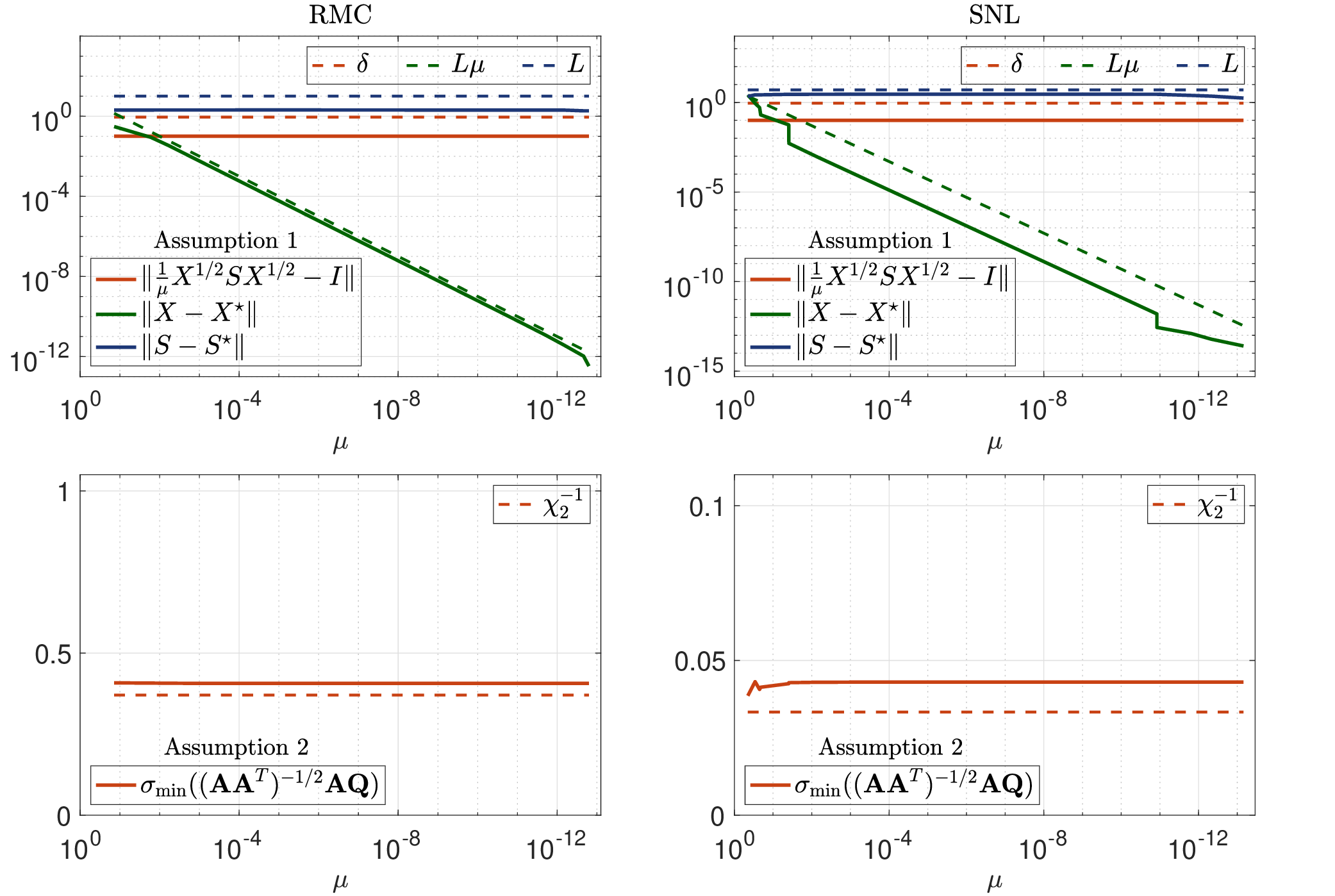}
    \vspace{-2em}
	\caption{Verifying Assumption~\ref{asm:analy}: $\|\frac{1}{\mu}X^{1/2}SX^{1/2}-I\|\leq \delta$, $\|X-X^\star\|\leq L\mu$, $\|S-S^\star\|\leq L$, and Assumption~\ref{asm:inject}: $\min_{x} \|(\A\A^T)^{-1/2}\A\Q x\|/\|\Q x\|\geq \chi_2^{-1}$. \textbf{Left.} Both assumptions hold for RMC with $(\delta,L,\chi_2)=(0.9,10,2.7)$. \textbf{Right.} Both assumptions hold for SNL with $(\delta,L,\chi_2)=(0.9,5,30)$.}\label{fig:asm}
    \vspace{-0.5em}
\end{figure} 

\paragraph{Bounded condition number and \texttt{PCG} iterations.}
As the two assumptions hold for both problems, Theorem~\ref{thm:main} proves that the condition number of the corresponding indefinite system (\ref{eq:augrec}) is bounded and independent of $\mu$. Additionally, Corollary~\ref{cor:pcg} shows that the indefinite system can be efficiently solved using \texttt{PCG}, and the number of iterations required to satisfy residual condition (\ref{eq:Newt-eps}) is independent of $1/\mu$.

In Figure~\ref{fig:thm}, we plot the condition number of our indefinite system (\ref{eq:augrec}) and the KKT equations (\ref{eq:Newt-eps}), as well as the number of \texttt{PCG} iterations required to satisfy residual condition (\ref{eq:Newt-eps}) with respect to different values of $\mu$. We use the same instance of RMC (\ref{eq:rmc-sdp}) and SNL (\ref{eq:snl-sdp}) in Figure~\ref{fig:asm}. For each problem, we extract $W$ and $\mu$ as in Figure~\ref{fig:asm}. For each pair of $W$ and $\mu$, the corresponding indefinite system is constructed as in (\ref{eq:augrec}) using the rank parameter $r=\arg\max_{i\leq 5}\lambda_i(W)/\lambda_{i+1}(W)$, and the corresponding KKT equations are constructed as in (\ref{eq:Newt-eps}). We run \texttt{PCG} to solve the indefinite system until the residual condition (\ref{eq:Newt-eps}) is satisfied with $\epsilon=10^{-6}$. In both cases, we see that the condition number of the indefinite system is upper bounded, while the condition number of the KKT equations (\ref{eq:Newt-eps}) explodes as $\mu\to 0$. (We note that it becomes stuck at $10^{25}$ due to finite precision.) The number of \texttt{PCG} iterations for satisfy residual condition (\ref{eq:Newt-eps}) is also bounded and independent of $1/\mu$.

\begin{figure}[b!]
	\center
	\vspace{-0.5em}
    \includegraphics[width=\textwidth]{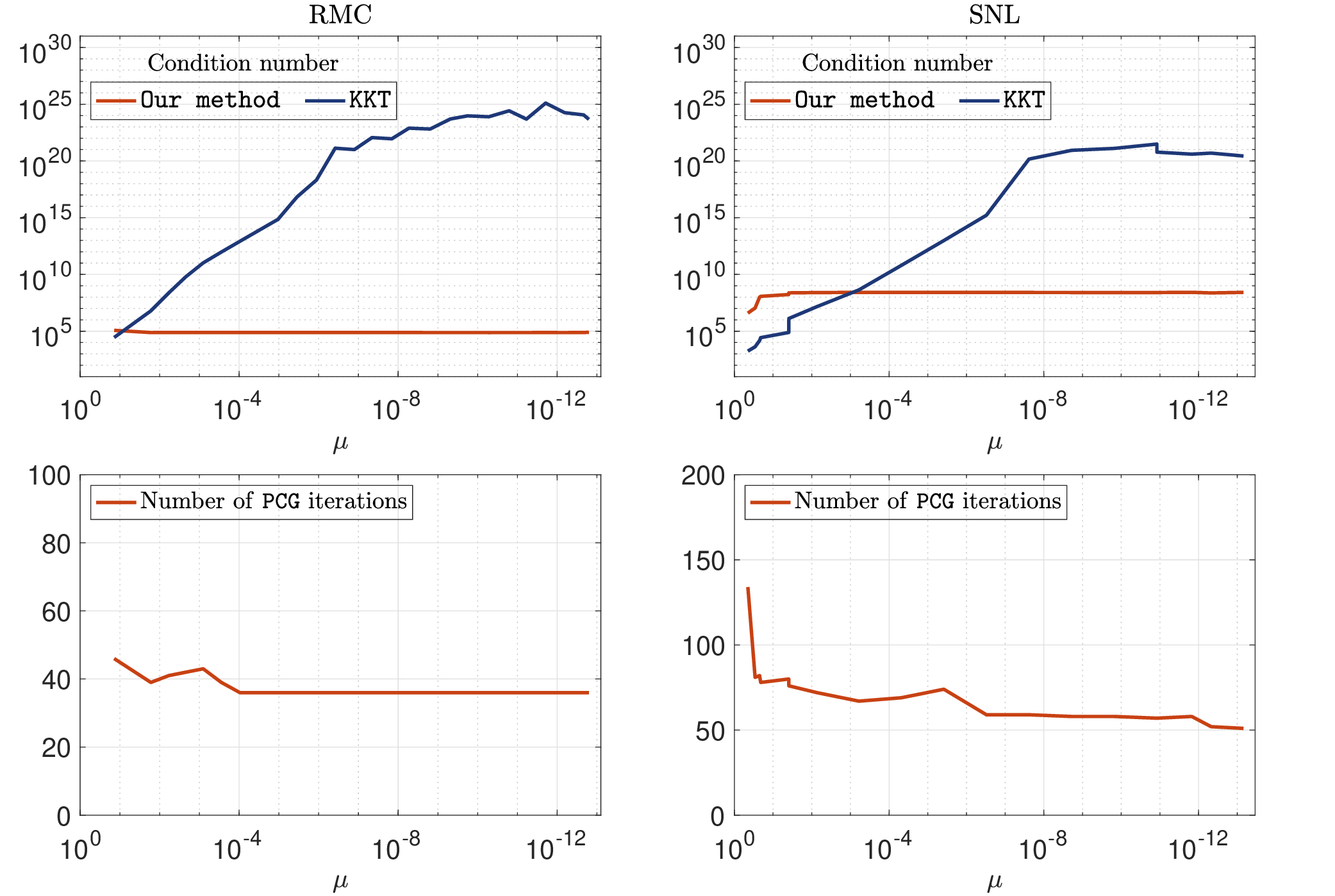}
    \vspace{-2em}
	\caption{Verifying Theorem~\ref{thm:main}: the condition number of the indefinite system~(\ref{eq:augrec}) is bounded and independent of $\mu$, and Corollary~\ref{cor:pcg}: the number of \texttt{PCG} iterations required to satisfy residual condition (\ref{eq:Newt-eps}) is bounded and independent of $1/\mu$. \textbf{Left.} RMC. \textbf{Right.} SNL.}\label{fig:thm}
	\vspace{-0.5em}
\end{figure}

\paragraph{Cubic time complexity.} Finally, we verify \texttt{Our method} achieves cubic time complexity when Assumption~\ref{asm:mvp2} holds and $r\ll n$, with no dependence on the number of constraints $m$. We first prove both RMC (\ref{eq:rmc-sdp}) and SNL (\ref{eq:snl-sdp}) satisfy Assumption~\ref{asm:mvp2}. Given any $U,V\in\R^{n\times r}$ and $y\in\R^m$. For RMC, it is obvious that both $\AA(UV^T+VU^T)=P_{m}\svec(UV^T+VU^T)$ and $\svec(\AA^T(y))=P_m^Ty$ can be evaluated in $O(n^2r)$ time and $O(n^2)$ storage. For SNL, by first paying $O(kr^2)$ time to calculate $a_1a_1^T,\ldots,a_ka_k^T$, observe that 

\begin{gather*}
    \begin{bmatrix}0\\e_i-e_j\end{bmatrix}^T\left(\begin{bmatrix}I\\U\end{bmatrix}\begin{bmatrix}I\\V\end{bmatrix}^T+\begin{bmatrix}I\\V\end{bmatrix}\begin{bmatrix}I\\U\end{bmatrix}^T\right)\begin{bmatrix}0\\e_i-e_j\end{bmatrix},\quad \begin{bmatrix}0\\e_i-e_j\end{bmatrix}\begin{bmatrix}0\\e_i-e_j\end{bmatrix}^T\\
    \begin{bmatrix}a_i\\-e_j\end{bmatrix}^T\left(\begin{bmatrix}I\\U\end{bmatrix}\begin{bmatrix}I\\V\end{bmatrix}^T+\begin{bmatrix}I\\V\end{bmatrix}\begin{bmatrix}I\\U\end{bmatrix}^T\right)\begin{bmatrix}a_i\\-e_j\end{bmatrix},\quad \begin{bmatrix}a_i\\-e_j\end{bmatrix}\begin{bmatrix}a_i\\-e_j\end{bmatrix}^T
\end{gather*}
can all be evaluated in $O(r)$ time. (Notice that $r=d$ in SNR.) Therefore, when $k\ll n$, which is common in real-world applications where there are typically more sensors than anchors, both $\AA(X)$ and $\AA^T(y)U$ can be evaluated in $O(n^2r)$ time and $O(n^2)$ storage.
\begin{figure}[b!]
	\center
    \vspace{-0.5em}
    \includegraphics[width=\textwidth]{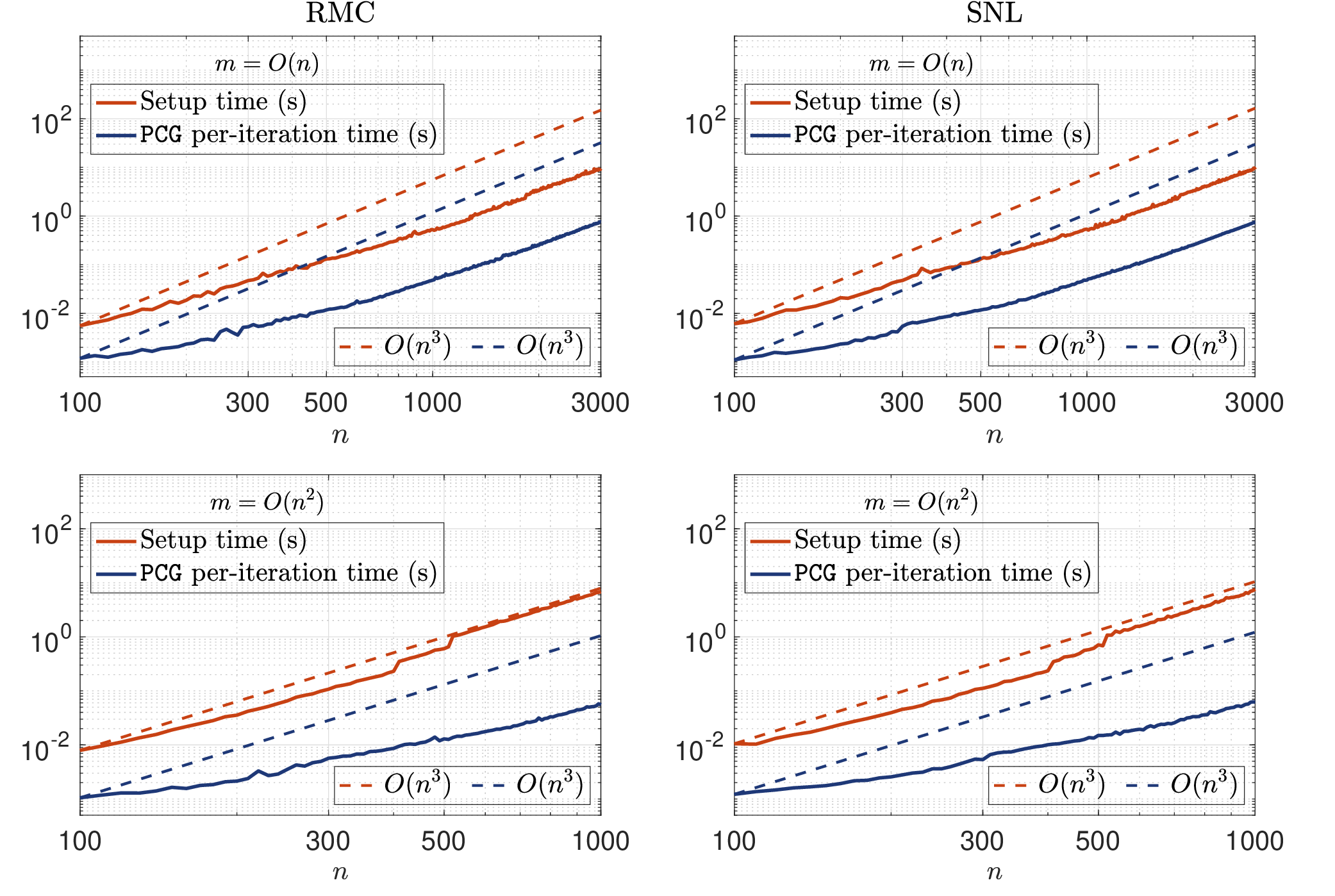}
    \vspace{-2em}
	\caption{Verifying cubic time complexity for \texttt{Our method}: both setting up the indefinite system (\ref{eq:augrec}) and one iteration of \texttt{PCG} cost $O(n^3)$ time, with no dependency on the number of constraints $m$. \textbf{Left.} RMC with $m=O(n)$ and $m=O(n^2)$. \textbf{Right.} SNL with $m=O(n)$ and $m=O(n^2)$.}\label{fig:runtime}
	\vspace{-0.5em}
\end{figure}

In Figure~\ref{fig:runtime}, we plot the average setup time for the indefinite system (\ref{eq:augrec}) and the average per \texttt{PCG} iteration time against $n$ for both RMC and SNL problem. Specifically, the setup time includes: the time to decompose $W=Q\Lambda Q^T+Q_\perp\Lambda_\perp Q_\perp^T$ and construct $E=QQ^T+\tau^{-1}Q_\perp\Lambda_\perp Q_\perp^T$ in order to implicitly setup $\Q$, $\E$, and $\Sig$ as in (\ref{eq:EQSigdef}); explicitly form the Schur complement $\C = \tau^2\Sig + \beta^{-1}\Q^T\A^T\A\Q$ via $O(nr)$ matrix-vector product with $\A\Q$ and $\Q^T\A^T$; and compute the Cholesky factor of $\C$. We consider two cases, $m=O(n)$ and $m=O(n^2)$, in order to verify the time complexity of \texttt{Our method} is indeed invariant to $m$. For both problem, we set the number of constraints to be $m=20n$ for the case of $m=O(n)$, and $m=n(n-1)/2$ for the case of $m=O(n^2)$. As we are interested in measuring the run time of \texttt{Our method} for solving the standard form SDP (\ref{eq:sdp}), we fix $m_o=0$ to eliminate all the LP constraint in (\ref{eq:rmc-sdp}) and (\ref{eq:snl-sdp}). As shown in Figure~\ref{fig:runtime}, the setup time and per \texttt{PCG} time exhibit cubic time complexity for both problems.

\subsection{Comparison against other approaches}\label{sec:comparison}
\begin{table}[!b]
\footnotesize
\vspace{-1em}
\caption{Reconstruction error (Error) and runtime (Time) for solving rank-$2$ size $n\times n$ RMC problems with $m$ constraints and $m_o$ outliers.}\label{table:rmc}
\begin{center}
\resizebox{0.95\linewidth}{!}{%
\begin{tabular}{c@{\hskip 4pt}c@{\hskip 4pt}c c@{\hskip 6pt}c c@{\hskip 6pt}c c@{\hskip 6pt}c c@{\hskip 6pt}c}
	\toprule  
	\multicolumn{1}{c}{\small$n$} & \multicolumn{1}{c}{\small$m$} & \multicolumn{1}{c}{\small$m_o$} & \multicolumn{2}{c}{\small\texttt{Our method}} & \multicolumn{2}{c}{\small\texttt{MOSEK}} & \multicolumn{2}{c}{\small\texttt{ZL17}} & \multicolumn{2}{c}{\small\texttt{ADMM}}\\
   \midrule
  && & \footnotesize{Error} & \footnotesize{Time} & \footnotesize{Error} & \footnotesize{Time}&\footnotesize{Error} & \footnotesize{Time}&\footnotesize{Error} & \footnotesize{Time}\\
	\cmidrule(lr){4-5}\cmidrule(lr){6-7}\cmidrule(lr){8-9}\cmidrule(lr){10-11}
    100 &  3000 &  12 & $\bmexpm{1.2}{12}$ & 1.2s  & $\expm{2.3}{10}$ & 2.5s & $\expm{2.7}{06}$ & 1.6s & $\expm{7.4}{07}$ & 24.6s\\
    100 &  4000 &  24 & $\bmexpm{1.5}{12}$ & 1.0s  & $\expm{2.3}{11}$ & 5.3s & $\expm{2.2}{06}$ & 1.7s & $\expm{8.4}{07}$ & 25.9s\\
    100 &  5000 &  48 & $\bmexpm{1.9}{12}$ & 2.2s  & $\expm{1.5}{08}$ & 6.8s & $\expm{9.6}{06}$ & 2.2s & $\expm{7.5}{07}$ & 152s\\
    \midrule
    300 & 40000 &  50 & $\bmexpm{3.1}{12}$ & 21.7s & $\expm{2.6}{09}$ & 1853s & $\expm{1.6}{06}$ & 31.3s & $\expm{7.0}{07}$ & 84.2s\\
    300 & 42500 & 100 & $\bmexpm{4.2}{12}$ & 23.9s & $\expm{2.7}{07}$ & 2218s & $\expm{3.9}{06}$ & 42.0s & $\expm{8.9}{07}$ & 48.5s\\
    300 & 45000 & 150 & $\bmexpm{5.0}{12}$ & 32.1s & $\expm{3.6}{11}$ & 2310s & $\expm{7.4}{07}$ & 67.9s & $\expm{8.5}{07}$ & 50.1s\\
    \midrule
    500 & 115000 & 100 & $\bmexpm{7.0}{12}$ & 127s & \multicolumn{2}{c}{\scriptsize\texttt{Out of memory}} & $\expm{7.5}{06}$ & 154s & $\expm{6.7}{07}$ & 103s \\
    500 & 120000 & 200 & $\bmexpm{7.9}{12}$ & 138s & \multicolumn{2}{c}{\scriptsize\texttt{Out of memory}} & $\expm{1.7}{06}$ & 260s & $\expm{7.6}{07}$ & 180s \\
    500 & 125000 & 300 & $\bmexpm{7.4}{12}$ & 127s & \multicolumn{2}{c}{\scriptsize\texttt{Out of memory}} & $\expm{6.2}{06}$ & 233s & $\expm{8.0}{07}$ & 314s \\
    \midrule
    1000 & 400000 & 150 & $\bmexpm{1.3}{11}$ & 1011s & \multicolumn{2}{c}{\scriptsize\texttt{Out of memory}} & $\expm{1.5}{06}$ & 2253s & $\expm{7.5}{07}$ & 522s \\
    1000 & 450000 & 300 & $\bmexpm{1.1}{11}$ &  954s & \multicolumn{2}{c}{\scriptsize\texttt{Out of memory}} & $\expm{9.7}{07}$ & 1675s & $\expm{2.2}{07}$ & 594s \\
    1000 & 500000 & 450 & $\bmexpm{1.2}{11}$ & 1106s & \multicolumn{2}{c}{\scriptsize\texttt{Out of memory}} & $\expm{2.1}{05}$ & 1458s & $\expm{9.4}{07}$ & 1027s \\
	\bottomrule
\end{tabular}
}
\end{center}
\vspace{-2.5em}
\end{table}
In this section, we compare the practical performance of \texttt{Our method} for solving both robust matrix completion and sensor network localization problems. 
\paragraph{Robust matrix completion.}
\sloppy First-order methods, such as linearized alternating direction method of multipliers (\texttt{ADMM}) \cite{yuan2009sparse,cherapanamjeri2017nearly}, can be used to solve (\ref{eq:rmc}) via the following iterations:
\begin{align}
\begin{split}
X_{k+1}&=\mathcal{P}_r(X_k-\alpha\nabla f_\lambda(X_k,s_k,y_k)),\\
s_{k+1}&= \mathcal{S}_{\lambda/\beta}(\AA(X_{k+1})+y_{k}-b),\\
y_{k+1}&= y_{k} + \beta(\AA(X_{k+1})-s_{k+1}-b)
\end{split}\label{eq:admm}
\end{align}
where $\alpha=\frac{n^2}{m}$, $\beta$ is the step size, $f_\lambda(X,s,y)=\frac{\beta}{2}\|\AA(X)-b-s+y\|^2+\lambda\tr(X)$, $\mathcal P_r(X)$ denotes the projection of $X$ onto a set of rank-$r$ matrices, and $\mathcal S_{\gamma}(x)=\max\{x-\gamma,0\}+\min\{x+\gamma,0\}$ is the soft-thresholding operator. While fast, the convergence of \texttt{ADMM} is not guaranteed without additional assumption on $X^\star$, such as requiring $X^\star$ to be incoherent in the context of matrix completion. This makes it less reliable for obtaining high accuracy solutions. 

Alternatively, the Burer--Monteiro (BM) approach can be applied by substituting $X=UU^T$ where $U\in\R^{n\times r}$ and $r\geq \rank(X^\star)$. Unfortunately, as we explained in the introduction, the Burer--Monteiro approach is not able to solve the problem satisfactorily as it would either solve the problem by substituting $X=UU^T$ into (\ref{eq:rmc}) or (\ref{eq:rmc-sdp}). The former is an unconstrained nonsmooth problem, which can be solved via subgradient (\texttt{SubGD}) method $U_{k+1}=U_k-\alpha\nabla f(U_k)$ with step-size $\alpha\in(0,1]$ and $f(U)=\|\AA(UU^T)-b\|_1+\lambda\|U\|_F^2$. The latter is a constrained twice-differentiable problem, which can be solved via nonlinear programming (NLP) solver like \texttt{Fmincon}~\cite{matlabOTB} and \texttt{Knitro}~\cite{byrd2006k}. But in both cases, local optimization algorithms can struggle to converge, because the underlying nonsmoothness of the problem is further exacerbated by the nonconvexity of the BM factorization.

A promising alternative is to solve (\ref{eq:rmc-sdp}) via general-purpose IPM solver such as \texttt{MOSEK} \cite{mosek2015}. However, general-purpose IPM incurs at least $\Omega((n+m)^3)$ time and $\Omega((n+m)^2)$ memory, which is concerning for problem instances with $m=O(n^2)$ measurements. 

Instead, we present experimental evidence to show that \texttt{Our method} is the most reliable way of solving RMC to high accuracies. We start by demonstrating the efficiency and efficacy of \texttt{Our method} by solving a large-scaled RMC, which is generated with $(n,m,m_o)=(500,125250,250)$. In Figure~\ref{fig:main}, we plot the reconstruction error $\|X-X^\star\|_F$ against runtime for \texttt{Our method}, \texttt{ZL17}, \texttt{Fmincon}, \texttt{Knitro}, \texttt{ADMM} and \texttt{SubGD}. Here, the maximum rank parameter $\hat r$ for \texttt{Our method}, and the search rank $r$ for \texttt{Fmincon}, \texttt{Knitro}, \texttt{ADMM} and \texttt{SubGD} are all set to be $r=10$. We initialized \texttt{Fmincon}, \texttt{Knitro} and \texttt{SubGD} at the same initial point $U_0$ that is drawn from standard Gaussian, and initialized \texttt{ADMM} at $X_0=U_0U_0^T$. The step-size of \texttt{ADMM} and \texttt{SubGD} is set be $\beta=10^{-2}$ and $\alpha=10^{-6}$, respectively. From Figure~\ref{fig:main}, we see that \texttt{Our method} is the only one that is able to achieve reconstruction error below $10^{-8}$. Though both \texttt{ADMM} and \texttt{ZL17} achieves reconstruction error $10^{-6}$ at the fastest rate, they cannot make any further progress; \texttt{ADMM} becomes stuck and \texttt{ZL17} becomes ill-conditioned. Comparing to nonlinear programming solvers, both \texttt{Fmincon} and \texttt{Knitro} eventually reach the same reconstruction error similar to \texttt{ZL17} but at a significantly slower rate than \texttt{Our method}. Finally, \texttt{SubGD} converges extremely slowly for this problem.

\begin{figure}[!b]
    \center
	\vspace{-1em}
    \includegraphics[width=1\textwidth]{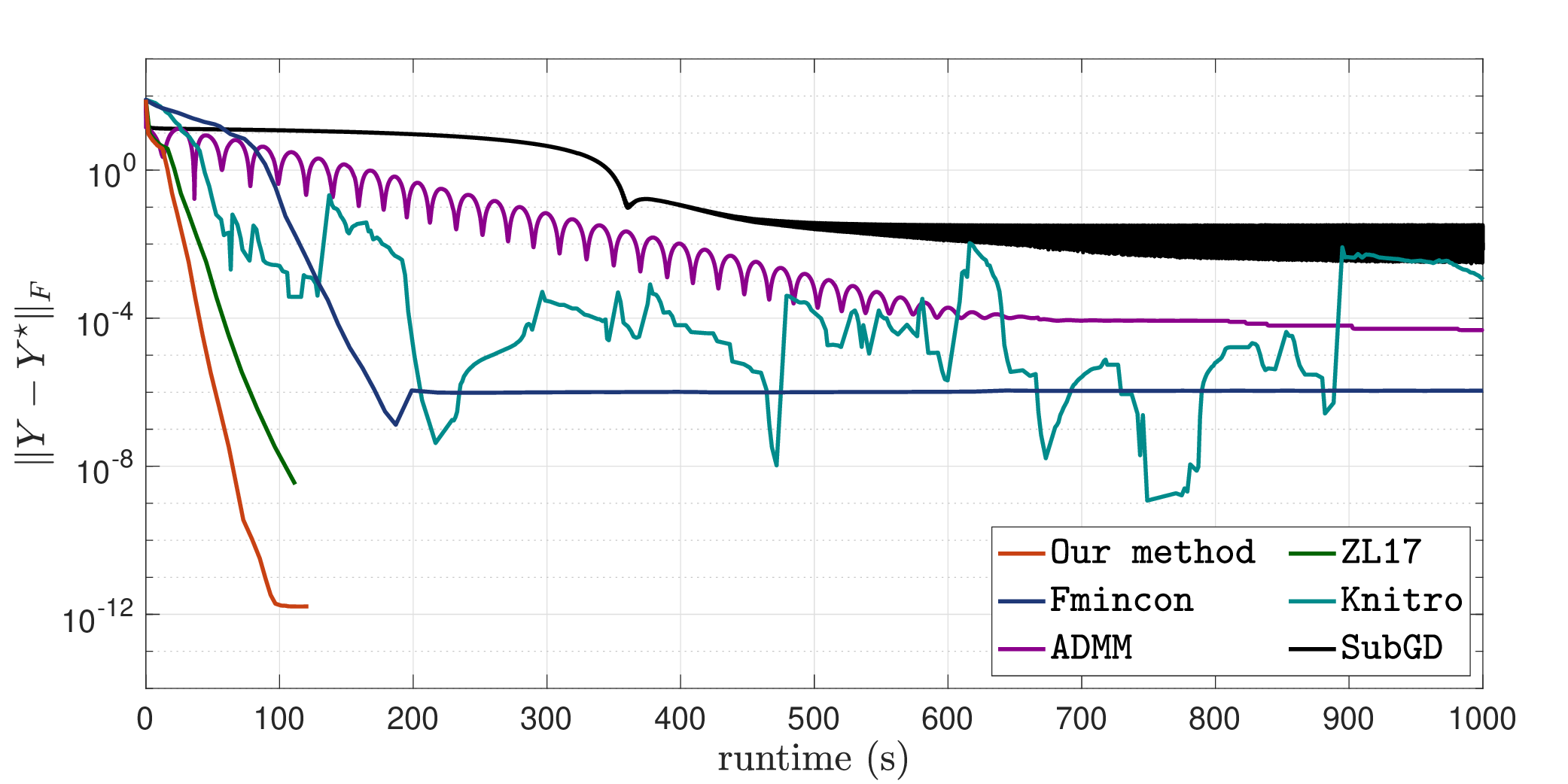}
    \caption{Reconstruction error $\|Y-Y^{\star}\|_{F}$ against runtime for solving an instance of SNL problem (\ref{eq:snl-sdp}) using \texttt{Our method}, \texttt{ZL17}, \texttt{Fmincon}, \texttt{Knitro}, \texttt{ADMM}, and \texttt{SubGD}. The SNL problem is generated with $(n,m,m_0)=(300,45000,300)$.}\label{fig:main_snl}
    \vspace{-1em}
\end{figure}

In Table~\ref{table:rmc}, we solve several instances of RMC problems using \texttt{Our method}, \texttt{MOSEK}, \texttt{ZL17} and \texttt{ADMM}, and report their corresponding reconstruction error and runtime. For space efficiency, we omit the comparison with \texttt{Knitro}, \texttt{Fmincon} and \texttt{SubGD} as they generally perform worse than \texttt{ADMM}. In this experiment, the maximum rank parameter $\hat r$ for \texttt{Our method} and the search rank $\hat r$ for \texttt{ADMM} are all set to be $10$. The number of iterations for \texttt{ADMM} is set to be $10^5$. For all four methods, we report the best reconstruction error and the time they achieve their best reconstruction error. From Table~\ref{table:rmc}, we see that \texttt{Our method} consistently achieves reconstruction error around $10^{-12}$ in all cases, with runtime similar to \texttt{ZL17}. For $n=500$ and larger, \texttt{MOSEK} runs out of memory. We note that we pick the largest possible step size for \texttt{ADMM} that allows it to stably converge to the smallest error.

\paragraph{Sensor network localization.}
We again demonstrate the efficiency and efficacy of \texttt{Our method} by solving a large-scale SNL problem, which is generated with $(n,m,m_o)=(300,45000,300)$. Similar to RMC, SNL can also be solved via \texttt{Fmincon}, \texttt{Knitro}, \texttt{ADMM} and \texttt{SubGD}. In particular, \texttt{ADMM} can be applied by first rewriting (\ref{eq:snl-sdp}) into the standard form $\min_{X\succeq 0} \|\AA(X)-b\|_1$, and then apply (\ref{eq:admm}) with $\lambda=0$. The \texttt{SubGD} method can subsequently be applied with the update $U_{k+1}=U_k-\alpha\nabla f(U_k)$, where $\alpha\in(0,1]$, $U\in\R^{n\times r}$, and $f(U)=\|\AA(UU^T)-b\|_1$. Finally, both \texttt{Fmincon} and \texttt{Knitro} can be applied by substituting $X=UU^T$ into (\ref{eq:snl-sdp}).

In Figure~\ref{fig:main_snl}, we plot the reconstruction error against runtime for \texttt{Our method}, \texttt{ZL17}, \texttt{Fmincon}, \texttt{Knitro}, \texttt{ADMM} and \texttt{SubGD}. The reconstruction error for SNL is defined as $\|Y-Y^\star\|_F$ where $Y^\star = [y_1^\star,\ldots,y_n^\star]$. Here, the maximum rank parameter $\hat r$ for \texttt{Our method}, and the search rank $r$ for \texttt{Fmincon}, \texttt{Knitro}, \texttt{ADMM} and \texttt{SubGD} are all set to be $r=10$. We initialized \texttt{Fmincon}, \texttt{Knitro} and \texttt{SubGD} at the same initial point $U_0$ that is drawn from standard Gaussian, and initialized \texttt{ADMM} at $X_0=U_0U_0^T$. The step-size of \texttt{ADMM} and \texttt{SubGD} is set be $\beta=0.03$ and $\alpha=5\times10^{-6}$, respectively. From Figure~\ref{fig:main_snl}, we again see that \texttt{Our method} is the only one that is able to achieve reconstruction error around $10^{-12}$ at the fastest rate.

In Table~\ref{table:snl}, we solve several instances of SNL problems using \texttt{Our method}, \texttt{MOSEK}, \texttt{ZL17} and \texttt{Knitro}, and report their corresponding reconstruction error and runtime. For simplicity, we omit the comparison with \texttt{ADMM}, \texttt{Fmincon}, and \texttt{SubGD} as they generally perform worse than \texttt{Knitro}. In this experiment, the maximum rank parameter $\hat r$ for \texttt{Our method} and the search rank $\hat r$ for \texttt{Knitro} are all set to be $10$. The number of iterations for \texttt{Knitro} is set to be $500$. For all four method, we report the best reconstruction error and the time they achieve their best reconstruction error. As shown in Table~\ref{table:snl}, \texttt{Our method} again consistently achieves reconstruction error around $10^{-12}$ in all cases, with runtime similar to \texttt{ZL17}. For $n=500$ and larger, \texttt{MOSEK} runs out of memory.

\begin{table}[!hb]
\footnotesize
\caption{Reconstruction error (Error) and runtime (Time) for solving 2-dimensional SNL problems with $n$ sensors, 3 anchors, $m$ distance measurements and $m_o$ outliers.}\label{table:snl}
\begin{center}
\resizebox{0.95\linewidth}{!}{%
\begin{tabular}{c@{\hskip 4pt}c@{\hskip 4pt}c c@{\hskip 6pt}c c@{\hskip 6pt}c c@{\hskip 6pt}c c@{\hskip 6pt}c}
	\toprule  
	\small$n$ & \small$m$ & \small$m_o$ & \multicolumn{2}{c}{\small\texttt{Our method}} & \multicolumn{2}{c}{\small\texttt{MOSEK}} & \multicolumn{2}{c}{\small\texttt{ZL17}} & \multicolumn{2}{c}{\small\texttt{Knitro}}\\
    \midrule
    &&& \footnotesize{Error} & \footnotesize{Time} & \footnotesize{Error} & \footnotesize{Time}&\footnotesize{Error} & \footnotesize{Time} & \footnotesize{Error} & \footnotesize{Time}\\
	\cmidrule(lr){4-5}\cmidrule(lr){6-7}\cmidrule(lr){8-9}\cmidrule(lr){10-11}
     50 &   1000 &  30 & $\bmexpm{2.9}{13}$ & 1.0s & $\expm{1.1}{09}$ & 0.6s & $\expm{5.6}{10}$ & 1.1s & $\expm{4.5}{10}$ & 7.0s \\
     50 &   1150 &  40 & $\bmexpm{2.4}{13}$ & 0.7s & $\expm{9.3}{13}$ & 0.5s & $\expm{1.2}{09}$ & 0.8s & $\expm{4.3}{10}$ & 9.2s \\
     50 &   1300 &  50 & $\bmexpm{1.7}{13}$ & 0.6s & $\expm{4.5}{12}$ & 0.7s & $\expm{8.8}{10}$ & 0.8s & $\expm{4.6}{10}$ & 14.6s \\
    \midrule
	100 &   3000 &  60 & $\bmexpm{4.3}{13}$ & 2.2s & $\expm{1.2}{12}$ & 3.2s & $\expm{3.0}{09}$ & 3.1s & $\expm{8.4}{09}$ & 64.4s \\
    100 &   4000 &  80 & $\bmexpm{8.5}{13}$ & 2.9s & $\expm{1.1}{10}$ & 4.4s & $\expm{3.8}{09}$ & 3.4s & $\expm{3.9}{11}$ & 53.7s \\
    100 &   5000 & 100 & $\bmexpm{4.2}{13}$ & 4.1s & $\expm{1.3}{10}$ & 7.8s & $\expm{1.0}{08}$ & 3.8s & $\expm{2.1}{08}$ & 59.2s \\
    \midrule
    300 &  40000 & 220 & $\bmexpm{2.2}{12}$ & 108.5s & $\expm{5.6}{10}$ & 1801s & $\expm{1.8}{09}$ & 136.2s & $\expm{1.5}{09}$ & 415.4s \\
    300 &  42500 & 260 & $\bmexpm{1.7}{12}$ &  95.3s & $\expm{3.7}{11}$ & 2439s & $\expm{1.5}{09}$ & 119.3s & $\expm{2.7}{09}$ & 275.9s \\
    300 &  45000 & 300 & $\bmexpm{1.6}{12}$ & 117.5s & $\expm{2.5}{11}$ & 3697s & $\expm{3.3}{09}$ & 111.7s & $\expm{1.2}{09}$ & 749.0s\\
    \midrule
    500 & 100000 & 420 & $\bmexpm{5.9}{12}$ & 499.3s & \multicolumn{2}{c}{\scriptsize\texttt{Out of memory}} & $\expm{1.5}{09}$ & 592.8s & $\expm{9.9}{10}$ & 2548s \\
    550 & 110000 & 460 & $\bmexpm{2.6}{12}$ & 559.4s & \multicolumn{2}{c}{\scriptsize\texttt{Out of memory}} & $\expm{6.8}{09}$ & 747.7s & $\expm{4.1}{09}$ & 1081s \\
    500 & 120000 & 500 & $\bmexpm{6.0}{12}$ & 760.5s & \multicolumn{2}{c}{\scriptsize\texttt{Out of memory}} & $\expm{2.1}{08}$ & 955.5s & $\expm{1.2}{09}$ & 595.0s\\
	\bottomrule
\end{tabular}
}
\end{center}
\end{table}

\newpage
\bibliographystyle{siam}
\bibliography{arxiv.bib}

\newpage
\appendix

\section{Proof of Proposition~\ref{prop:correct}} \label{app:correctness}

We break the proof into three parts. Below, recall that we have defined
$\Psi_{n}$ as the basis matrix for $\vector(\S^{n})$ on $\vector(\R^{n\times n})=\R^{n^{2}}$,
in order to define $\svec(X)\eqdef\Psi_{n}^{T}\vector(X)=\frac{1}{2}\Psi_{n}^{T}\vector(X+X^{T})$
for $X\in\R^{n\times n}$, and $A\skron B=B\skron A\eqdef\frac{1}{2}\Psi_{n}^{T}(A\otimes B+B\otimes A)\Psi_{r}$
for $A,B\in\R^{n\times r}$.
\begin{lemma}[Correctness]
Let $W=Q\Lambda Q^{T}+Q_{\perp}\Lambda_{\perp}Q_{\perp}^{T}$ where
$[Q,Q_{\perp}]$ is orthonormal, and pick $0<\tau<\lambda_{\min}(\Lambda)$.
Define $E=QQ^{T}+\tau^{-1}\cdot Q_{\perp}\Lambda_{\perp}Q_{\perp}^{T},$
$\Q_{B}=Q\skron Q,$ $\Q_{N}=\sqrt{2}\Psi_{n}^{T}(Q\otimes Q_{\perp}),$
$\Sig_{B}^{-1}=\Lambda\skron\Lambda-\tau^{2}I,$ and $\Sig_{N}^{-1}=(\Lambda-\tau I)\otimes\Lambda_{\perp}$.
Then 
\[
W\skron W=[\Q_{B},\Q_{N}]\begin{bmatrix}\Sig_{B}\\
 & \Sig_{N}
\end{bmatrix}^{-1}[\Q_{B},\Q_{N}]^{T}+\tau^{2}\cdot E\skron E.
\]
\end{lemma}

\begin{proof}
Write $W_{X}=\tau\cdot Q(\Lambda-\tau I)Q^{T}$ and $W_{S}=\tau^{-1}Q_{\perp}\Lambda_{\perp}Q_{\perp}^{T}$.
For arbitrary $Y\in\S^{n}$, it follows from $W=\tau E+\tau^{-1}W_{X}$
that
\begin{align*}
\Psi_{n}^{T}(W\otimes W)\Psi_{n}\svec(Y)= & \svec((\tau E+\tau^{-1}W_{X})Y(\tau E+\tau^{-1}W_{X}))\\
= & \svec\left[\tau^{2}\cdot EYE+2\cdot EYW_{X}+\tau^{-2}\cdot W_{X}YW_{X}\right].
\end{align*}
Now, substituting $E=QQ^{T}+W_{S}$ yields our desired claim:
\begin{align*}
= & \svec\left[\tau^{2}\cdot EYE+2\cdot W_{S}YW_{X}+\left(2QQ^{T}+\tau^{-2}W_{X}\right)YW_{X}\right]\\
= & \tau^{2}\cdot\underbrace{\svec(EYE)}_{\E y}+\underbrace{\svec(2W_{S}YW_{X})}_{\Q_{N}\Sig_{N}^{-1}\Q_{N}^{T}y}+\underbrace{\svec\left[\left(2QQ^{T}+\tau^{-2}W_{X}\right)YW_{X}\right]}_{\Q_{B}\Sig_{B}^{-1}\Q_{B}^{T}y}.
\end{align*}
\sloppy Indeed, we verify that $\vector(2W_{S}YW_{X})=(\sqrt{2}Q\otimes Q_{\perp})[(\Lambda-\tau I)\otimes\Lambda_{\perp}](\sqrt{2}Q\otimes Q_{\perp})^{T}\vector(Y)=\Q_{N}[(\Lambda-\tau I)\otimes\Lambda_{\perp}]\Q_{N}^{T}\vector(Y)$
and $\svec\left[\left(2QQ^{T}+\tau^{-2}W_{X}\right)YW_{X}\right]=\Q_{B}[(\Lambda+\tau I)\skron(\Lambda-\tau I)]\Q_{B}^{T}\svec(Y)$,
where we used $(A+B)\skron(A-B)=A\skron A-B\skron B$ because $A\skron B=B\skron A$. 
\end{proof}
\begin{lemma}[Column span of $\Q$]
\label{lem:spanQ}Let $[Q,Q_{\perp}]$ be orthonormal with $Q\in\R^{n\times r}$.
Define $\Q=[\Q_{B},\Q_{N}]$ where $\Q_{B}=Q\skron Q$ and $\Q_{N}=\sqrt{2}\Psi_{n}^{T}(Q\otimes Q_{\perp})$.
Then, $\colspan(\Q)=\{\svec(VQ^{T}):V\in\R^{n\times r}\}$ and $\Q^{T}\Q=I_{d}$
where $d=nr-\frac{1}{2}r(r-1)$.
\end{lemma}

\begin{proof}
Without loss of generality, let $[Q,Q_{\perp}]=I_{n}$. We observe
that
\[
\Q\begin{bmatrix}\svec(H_{1})\\
\vector(H_{2})
\end{bmatrix}=\svec\left(\begin{bmatrix}H_{1} & \frac{1}{\sqrt{2}}H_{2}^{T}\\
\frac{1}{\sqrt{2}}H_{2} & 0
\end{bmatrix}\right)=\svec\left(\begin{bmatrix}H_{1}\\
\sqrt{2}H_{2}
\end{bmatrix}Q^{T}\right)
\]
holds for any arbitrary (possibly nonsymmetric) $H_{1}\in\R^{r\times r}$
and $H_{2}\in\R^{(n-r)\times r}$, and that $\|\Q h\|=\|h\|$ for
all $h\in\R^{d}$, so $\Q^{T}\Q=I_{d}$. 
\end{proof}
\begin{lemma}[Eigenvalue bounds on $\E$ and $\Sig$]\label{lem:eigen}
Given $W\in\S_{++}^{n}$, write $w_{i}\equiv\lambda_{i}(W)$
for all $i\in\{1,2,\dots,n\}$. Choose $r\ge1$ and $\tau=\frac{1}{2}w_{r+1}$,
and define $\E=E\skron E$ where $E=QQ^{T}+\tau^{-1}\cdot Q_{\perp}\Lambda_{\perp}Q_{\perp}^{T},$
and $\Sig=\diag(\Sig_{B},\Sig_{N})$ where $\Sig_{B}^{-1}=\Lambda\skron\Lambda-\tau^{2}I,$
and $\Sig_{N}^{-1}=(\Lambda-\tau I)\otimes\Lambda_{\perp}$. Then,
\[
\begin{array}{ccccccc}
4 & \ge & \lambda_{\max}(\E) & \ge & \lambda_{\min}(\E) & \ge & w_{n}^{2}/w_{r+1}^{2},\\
w_{r+1}^{2}/(w_{n}w_{r}) & \ge & \tau^{2}\cdot\lambda_{\max}(\Sig) & \ge & \tau^{2}\cdot\lambda_{\min}(\Sig) & \ge & w_{n}^{2}/(4w_{1}^{2}).
\end{array}
\]
\end{lemma}

\begin{proof}
\sloppy We have $\lambda_{\max}(E)=\max\{1,\lambda_{\max}(\Lambda_{\perp})/\tau\}=\max\{1,2w_{r+1}/w_{r+1}\}\le2$ and $\lambda_{\max}(\E)=\lambda_{\max}^{2}(E)$.
Similarly, $\lambda_{\min}(E)=\min\{1,\lambda_{\min}(\Lambda_{\perp})/\tau\}=\min\{1,2w_{n}/w_{r+1}\}\ge w_{n}/w_{r+1}$ and $\lambda_{\min}(\E)=\lambda_{\min}^{2}(E)$.
We have $\tau^{-2}\lambda_{\max}^{-1}(\Sig)\ge(\frac{1}{2}w_{r}w_{n})/(\frac{1}{2}w_{r+1})^{2}\ge w_{r}w_{n}/w_{r+1}^{2},$
because $\lambda_{\max}^{-1}(\Sig)=\min\{\lambda_{\max}^{-1}(\Sig_{B}),\lambda_{\max}^{-1}(\Sig_{N})\}$
and $\lambda_{\max}^{-1}(\Sig_{B})=\lambda_{\min}(\Lambda\skron\Lambda-\tau^{2}I)\ge w_{r}^{2}-\frac{1}{4}w_{r+1}^{2}\ge\frac{3}{4}w_{r}^{2},$
while $\lambda_{\max}^{-1}(\Sig_{N})=\lambda_{\min}[(\Lambda-\tau I)\otimes\Lambda_{\perp}]\ge(w_{r}-\frac{1}{2}w_{r+1})w_{n}\ge\frac{1}{2}w_{r}w_{n}.$
Finally, we have $\tau^{-2}\lambda_{\min}^{-1}(\Sig)\le w_{1}^{2}/(\frac{1}{2}w_{r+1})^{2}\le4w_{1}^{2}/w_{r+1}^{2},$
because $\lambda_{\min}^{-1}(\Sig)=\max\{\lambda_{\min}^{-1}(\Sig_{B}),\lambda_{\min}^{-1}(\Sig_{N})\}\le w_{1}^{2}$.
\end{proof}

\end{document}